\documentclass[12pt]{article}

\usepackage{amsmath,amsfonts,amssymb,graphicx,amsthm,color}

\usepackage{cite}

\usepackage[hidelinks,hypertexnames=false]{hyperref}

\usepackage[left=1.2in,right=1.2in]{geometry}

\usepackage{mathrsfs}

\usepackage[bf,medium,raggedright]{titlesec}

\usepackage{autonum}

\DeclareMathOperator*{\Tr}{\mathrm{tr}}

\newcommand{\Lgen}{\mathcal{L}}

\newcommand{\Meas}{\mathcal{M}}

\newcommand{\R}{\mathbb{R}}

\DeclareMathOperator{\supp}{\mathrm{supp}}

\DeclareMathOperator{\Div}{\mathrm{div}}
\DeclareMathOperator{\Id}{\mathrm{Id}}
\usepackage{bbm}
\newcommand{\1}{\mathbbm{1}}

\newcommand{\ep}{\epsilon}

\newcommand{\tensor}{\otimes}

\newcommand{\E}{\mathbf{E}}
\renewcommand{\P}{\mathbf{P}}

\newcommand{\leqc}{\lesssim}

\newtheorem{thm}{Theorem}[section]
\newtheorem{prop}[thm]{Proposition}

\newtheorem{lem}[thm]{Lemma}
\newtheorem*{lem*}{Lemma}

\theoremstyle{definition}
\newtheorem{definition}[thm]{Definition}

\theoremstyle{remark}
\newtheorem{remark}[thm]{Remark}

\newcommand{\loc}{\mathrm{loc}}

\newcommand{\MMMsw}[1]{ [\Meas_{t,x,v}]_{*} }

\newcommand{\dee}{\mathrm{d}}
\newcommand{\ds}{\dee s}
\newcommand{\dt}{\dee t}

\newcommand{\dx}{\dee x}

\newcommand{\dy}{\dee y}

\begin{document}

\title{{\bf Renormalized Solutions to\\ Stochastic Continuity Equations\\ with Rough Coefficients}} 
\author{Sam Punshon-Smith}

\maketitle

\begin{abstract}
  We consider the stochastic continuity equation associated to an It\^{o} diffusion with irregular drift and diffusion coefficients. We give regularity conditions under which weak solutions are renormalized in the sense of DiPerna/Lions, and prove well-posedness in $L^p$. As an application, we furnish a new proof of renormalizability (hence uniqueness) of weak solutions to the stochastic continuity equation when the diffusion matrix is constant and the drift only belongs to $L^q_tL^p$, where $\frac{2}{q} + \frac{n}{p} <1$, without resorting to the regularity of the stochastic flow or a duality method.
\end{abstract}

\section{Introduction and Motivation}

This article concerns the It\^{o} stochastic continuity equation
\begin{equation}\label{eq:stoch-transport}
  \begin{aligned}
    &\partial_tf + \Div(b f) - \frac{1}{2}\partial_i\partial_j(\sigma^k_i\sigma^k_j f) + \Div(\sigma^k f) \dot{W}^k = 0\\
    &f|_{t=0} = f_0,
   \end{aligned}
\end{equation}
on $[0,T]\times\R^n$, where $b:[0,T] \times \R^n \to \R^n$, $\sigma^k:[0,T]\times\R^n \to \R^n$, $k = 1,\ldots n$ are vector fields with some degree of irregularity (typically Sobolev or worse), $W^k$, $k=1\ldots n$ are a family of independent Wiener processes defined on a complete filtered probability space $(\Omega,\mathcal{F},\mathcal{F}_t,\P)$ and the stochastic integral interpreted in the It\^{o} sense. In the above equation, we follow the standard convention that repeated indices are summed.

Equation (\ref{eq:stoch-transport}) naturally arises from the following It\^{o} stochastic differential equation
\begin{equation}\label{eq:SDE}
  \dee X_t = b(t,X_t)\dt + \sigma^{k}(t,X_t)\dee W^k_t.
\end{equation}
If $b$ and $\sigma^k$ are smooth enough, then (\ref{eq:SDE}) admits a stochastic flow of diffeomorphisms (see \cite{Kunita1986-sr}) $x\mapsto X_t(x)$, whereby a simple consequence of It\^{o}'s formula and change of variables implies that
\begin{equation}\label{eq:pushforward-rep}
  f(t,x) = (X_{t}){}_{\#}f_0(x) = f_0(X_t^{-1}(x))\det\left(\nabla X_t^{-1}(x)\right)
\end{equation}
satisfies equation (\ref{eq:stoch-transport}) in the sense of distribution with initial data $f_0$. In fact, if the initial data $f_0$ is also smooth, then one can apply the classical results of Kunita \cite{Kunita1986-sr} to show that (\ref{eq:pushforward-rep}) actually defines the unique classical solution to (\ref{eq:stoch-transport}). 

In the rough (non-Lipschitz) setting, the results of \cite{Kunita1986-sr} do not apply and the validity of the representation formula (\ref{eq:pushforward-rep}) becomes unclear without some degree of regularity on $X_t(x)$. 

It is well known that for deterministic continuity equation
\begin{equation}\label{eq:determinisitic-eq}
  \partial_t f + \Div(bf) = 0,
\end{equation}
if $b$ does not have some degree of regularity, then the equation can be shown to be ill-posed \cite{Di_Perna1989-pu,Ambrosio2004-fx,Flandoli2009-pa}. At this level, it is convenient to work directly with the equation (\ref{eq:determinisitic-eq}), though one can also work with the flow (see \cite{Crippa2008-bk}). This was the approach taken in the seminal work of DiPerna-Lions \cite{Di_Perna1989-pu}, who introduced the notion of a {\it renormalized solution}, that is, a solution $f$ that satisfies
\begin{equation}\label{eq:determinisitic-eq-renormalized-form}
  \partial_t\Gamma(f) + \Div(b\, \Gamma(f)) + \Div b\,G(f) = 0,
\end{equation}
for all bounded $C^1(\R)$ functions $\Gamma(z)$, and $G(z) = z\Gamma^\prime(z) - \Gamma(z)$. Here, one of the main results in \cite{Di_Perna1989-pu} is the {\it renormalizability} of (\ref{eq:determinisitic-eq}), that is, if $b$ satisfies
\begin{equation}
  b \in L^1([0,T],W^{1,p}_\loc)\quad \text{and}\quad \Div b \in L^1_{\loc}([0,T]\times\R^n),
\end{equation}
Then any weak solution $f \in L^\infty([0,T];L^{r}_\loc)$, $r\leq \frac{p}{p-1}$ also satisfies (\ref{eq:determinisitic-eq-renormalized-form}) in the sense of distribution. This is achieved by the means of a so-called {\it commutator lemma} describing the decay of the the difference between a smoothing operation by convolution and the differential operator $ b\cdot\nabla$. A non-trivial extension to the case of vector fields with spatial regularity in $BV_{\loc}$ instead of $W^{1,1}_\loc$ was given by Ambrosio \cite{Ambrosio2004-fx}.

The technique of renormalization has proven to be a powerful one for establishing uniqueness and apriori estimates. For instance, one can show, as a consequence, that if, additionally one has
\begin{equation}
  b(1+ |x|)^{-1} \in L^1([0,T];L^\infty)
\end{equation}
then the solution $f$ to (\ref{eq:determinisitic-eq}) is unique.

This strategy can be carried over to the case of the stochastic transport equation (\ref{eq:stoch-transport}). Specifically, the case $\sigma^k = e_k$ was studied in \cite{Attanasio2011-iy}. Under the same conditions on $b$ as in the deterministic theory (spatial regularity in $BV_\loc$), the authors of \cite{Attanasio2011-iy} showed that (\ref{eq:stoch-transport}) can be renormalized, hence
\begin{equation}
  \partial_t\Gamma(f) + \Div(b\, \Gamma(f)) - \frac{1}{2}\Delta \Gamma(f) + \partial_k\Gamma(f) \dot{W}^k + \Div b\, G(f) = 0,
\end{equation}
for a suitable class of $\Gamma$. In addition, the main contribution of \cite{Attanasio2011-iy} was to show that renormalization implies uniqueness {\it without} the $L^\infty$ assumption on $\Div b$. This is done through a new uniqueness result on parabolic equations in $L^\infty$.

To the author's knowledge, there have been no attempts to prove renormalizability when $\sigma^k$ is rough and (potentially) non-degenerate. This has applications in the theory of turbulence, particularly the Kraichnan model \cite{Kraichnan1994-kc}, as well as in the modeling of complex and polymeric fluids, as remarked in \cite{Bris2008-fc}. One of the results in this paper addresses this question. In it's simplest form, we have
\begin{thm}\label{thm:rough-diff-renorm-theorm}
  Let $f$ be a weak $L^p$, $p \in [2,\infty]$ solution to the equation
  \begin{equation}\label{eq:simplified-stoch-trans}
    \partial_tf - \frac{1}{2}\partial_i\partial_j(\sigma_i\sigma_j f) + \Div(\sigma f) \dot{W} = 0.
  \end{equation}
  If $\sigma$ satisfies
  \begin{equation}
    \sigma \in L^2_t([0,T];W^{1,\gamma}_{\loc}), \quad \text{for}\quad \gamma = \frac{2p}{p-2},
  \end{equation}
  then $\Gamma(f)$ is a weak solution to the renormalized equation
  \begin{equation}\label{eq:thm-renorm-form}
\begin{aligned}
  &\partial_t\Gamma(f) - \frac{1}{2}\partial_i\partial_j(\sigma_i\sigma_j\Gamma(f)) + \Div(\sigma\Gamma(f))\dot{W}+ G(f)\Div\sigma \dot{W}\\
  &\hspace{1in}- \Div\big(\sigma(\Div\sigma) G(f)\big) = \frac{1}{2}G(f)\partial_i\sigma_j\partial_j\sigma_i + \frac{1}{2}H(f)(\Div\sigma)^2 , 
\end{aligned}
\end{equation}
for each $C^2(\R)$ and bounded $\Gamma(z)$, where
\begin{equation}
  G(z) = z\Gamma^\prime(z) - \Gamma(z),\quad H(z) = zG^\prime(z) - G(z).
\end{equation}
\end{thm}
There is an analogue between Theorem \ref{thm:rough-diff-renorm-theorm} and the renormalizability results of Lions/LeBris \cite{Bris2008-fc} for parabolic equations with rough diffusion coefficients. The main difference is that in \cite{Bris2008-fc}, some apriori regularity on the solution is required, whereas for Theorem \ref{thm:rough-diff-renorm-theorm} no more regularity than $L^p$ on the solution is required. Also, as in \cite{Bris2008-fc}, an easy condition for existence and uniqueness (analogous to $L^\infty$ bounds on the divergence) is
\begin{equation}
  \partial_i\sigma_j\partial_j\sigma_i\in L^1([0,T]; L^\infty),\quad (\Div\sigma)^2 \in L^1([0,T]; L^\infty).
\end{equation}

In general, one might expect that the addition of a (non-degenerate) stochastic term to the deterministic continuity equation would allow for an improved well-posedness theory. Indeed, this effect was first discovered in the case $\sigma^k = e_k$
\begin{equation}\label{eq:sig=ek}
  \partial_t f + \Div(b\nabla f) - \frac{1}{2}\Delta f + \partial_k f\dot{W}^k = 0
\end{equation}
by Flandoli/Gubinelli/Priola \cite{Flandoli2009-pa} using a combination of a regularizing effect on the flow $x\mapsto X_t(x)$ and a renormalization procedure to prove well-posedness for
\begin{equation}
  b \in L^1([0,T]; C^\alpha_b).
\end{equation}
Apart from this, there have been many works studying the regularity of the stochastic flow $X_t(x)$ associated to (\ref{eq:SDE}) with rough coefficients (see \cite{Malliavin1999-ao,Jan2002-je,Airault2002-ux,Zhang2010-mk,Flandoli2010-vu,Attanasio2010-lt,Zhang2011-bm,Fedrizzi2013-cz,Rezakhanlou2014-xl,Mohammed2015-fp,MR3572321}, among other). In these cases, a {\it regularizing} property of the SDE is used to produce a stochastic flow, which is {\it more} regular that the coefficients. 
Roughly speaking, when the stochastic flow $x\mapsto X_t(x)$ has enough regularity, one can justify a representation of the form (\ref{eq:pushforward-rep}) and use this to prove well-posedness of the stochastic transport equation. This approach has been taken, for instance, in \cite{Fedrizzi2013-qe,Flandoli2009-pa,Mohammed2015-fp,Neves2016-ux}.

There is a natural question, however, as to whether one can prove renormalizability of (\ref{eq:stoch-transport}) when $b$ has less regularity than would be required for the deterministic case, {\it without} showing regularity of the stochastic flow. Indeed, in \cite{Attanasio2011-iy}, the authors were unable to remove the $BV_{\loc}$ requirement on $b$ when renormalizing. 

In this paper we will use Theorem \ref{thm:rough-diff-renorm-theorm} to consider drifts of the type
\begin{equation}\label{eq:b-KR}
  b \in L^q([0,T];L^p), \quad \frac{2}{q} + \frac{n}{p} < 1,\quad q\geq 2, p > 2,
\end{equation}
as a result we obtain a renormalizability result, which in it's simplest form is as follows.
\begin{thm}\label{thm:main-theorem}
  Let $b$ satisfy (\ref{eq:b-KR}) and  let $f$ be a weak $L^{p^\prime}$ solution $p^\prime = \frac{p}{p-2}$ to (\ref{eq:sig=ek}). If, in addition, $b$ satisfies
  \begin{equation}
    \Div b \in L^1([0,T]; L^1_\loc),
  \end{equation}
  then $f$ satisfies the renormalized equation
  \begin{equation}\label{eq:weak-renorm-thm}
    \partial_t\Gamma(f) + \Div(b\, \Gamma(f)) - \frac{1}{2}\Delta \Gamma(f) + \partial_i\Gamma(f)\dot{W}_i +  G(f)\!\Div b = 0,
  \end{equation}
for each $\Gamma \in C^2_b(\R)$.
\end{thm}

Condition (\ref{eq:b-KR}) is reminiscent of the so-called Ladyzhenskaya-Prodi-Serrin condition of fluid mechanics, where $<$ is replaced by $\leq$. The case with strict inequality was first considered by Krylov and R\"{o}ckner \cite{Krylov2004-uu}, where they proved strong solutions to the SDE with such a drift. As such, we will refer to condition (\ref{eq:b-KR}) on $b$ as the Krylov-R\"{o}ckner condition.

The proof Theorem \ref{thm:main-theorem} is inspired by the results in \cite{Krylov2004-uu}. In fact, we will utilize a version of approach introduced in \cite{Flandoli2010-vu} used by \cite{Fedrizzi2013-cz}, where the authors study the regularity properties of the damped backwards parabolic problem
\begin{equation}\label{eq:back-para-ito-tanaka}
  \partial_tu_{\lambda} + b\cdot \nabla u_{\lambda} + \frac{1}{2}\Delta u_{\lambda} = \lambda u_\lambda - b, \quad u_\lambda|_{t=T} =0.
\end{equation}
The solution $u_{\lambda}:[0,T]\times\R^n\to \R^n$ can then be used as a Zvonkin type transformation to transform the rough stochastic problem (\ref{eq:sig=ek}) into to one with a more regular drift, at the expense of introducing rough diffusion coefficients (the so-called It\^{o}-Tanaka trick). The transformed problem can then be renormalized by a generalization of Theorem \ref{thm:main-theorem} to the transformed equation. Using careful estimates on the solution to (\ref{eq:back-para-ito-tanaka}) as $\lambda \to \infty$. We can pass the limit as $\lambda \to \infty$ in renormalized version of the transformed problem to recover the original, renormalized equation (\ref{eq:weak-renorm-thm}).

The uniqueness implied by Theorem \ref{thm:main-theorem}, while better than \cite{Maurelli2011-wx}, is not better than the results of \cite{Beck2014-vq} where the authors are able to also treat the critical case $\frac{2}{q}+ \frac{n}{p} = 1$. In fact, to the author's knowledge \cite{Beck2014-vq} is the best result so far regarding uniqueness of (\ref{eq:sig=ek}). It is worth remarking that the uniqueness proof in this case is obtained by a duality method the authors do not obtain a renormalized equation (\ref{eq:weak-renorm-thm}). At this point, it is still an open question as to whether one can obtain renormalized solutions to (\ref{eq:sig=ek}) in the critical case. 

Besides uniqueness, another motivation for considering renormalizability of (\ref{eq:stoch-transport}) stems from kinetic theory, particularly the Boltzmann equation. Here renormalized solutions play an essential role in proving existence of solutions with general $L^1$ initial data (see \cite{DiPerna1989-qq}), but are not enough to prove uniqueness. From a physical perspective, it is of interest to consider the Boltzmann equation under the influence of a rough deterministic force. If the force is too rough, then the equation cannot be renormalized, preventing one from obtaining renormalized solutions to the equation. However, with the addition of another (degenerate) stochastic force, interpreted as environmental noise, the hope is that the equation can be renormalized and one can obtain renormalized martingale martingale solutions (see \cite{Punshon-Smith2016-vs}). In the collision-less setting (free stochastic kinetic transport) well-posedness has been proved in \cite{Fedrizzi2016-tr}.

\subsection{Overview}

The plan of the rest of the paper is as follows: In Section \ref{sec:prelim} we go over some preliminaries, deriving the renormalized form as it appears in Theorem \ref{thm:rough-diff-renorm-theorm}, and precisely defining the notion of weak $L^p$ solutions and the renormalized form. In Section \ref{sec:renorm-rough-diff} we are concerned with proving Theorem \ref{thm:rough-diff-renorm-theorm}. To do this, we first prove a series of regularization lemmas (similar to commutator lemmas) unique to this problem. At the end of this section we also give the proofs of existence and uniqueness. In Section \ref{sec:renorm-sol-int-drift} we aim to prove Theorem \ref{thm:main-theorem}. To do this we first study precise analytic properties of the parabolic problem (\ref{eq:back-para-ito-tanaka}), specifically how various norms of $u_{\lambda}$ scale with $\lambda$. Then we prove that we can transform a weak $L^p$ solution to a more regular one by a careful choice of test function. Finally we study the behavior of $u_\lambda$ as $\lambda \to 0$, and use it to relax the renormalized form of the transformed problem to the original one.

\section{Preliminaries}\label{sec:prelim}
\subsection{Renormalized Form}\label{sec:Renormalized-form}

The renormalized form (\ref{eq:thm-renorm-form}) in Theorem \ref{thm:rough-diff-renorm-theorm} is not immediately obvious even at formal level. To renormalize, one has to account for It\^{o} corrections and renormalize the second order derivative. Indeed there are some rather interesting cancellations that occur between the It\^{o} terms and those produces by renormalizing the second order Laplacian.

In this section we outline the formal procedure for renormalizing (\ref{eq:thm-renorm-form}) under the assumptions that $f$ and $\sigma$ are smooth. The procedure is enlightening, and sheds some insight into the cancellations that occur in the proof of Theorem \ref{thm:rough-diff-renorm-theorm}.

We start by writing equation (\ref{eq:simplified-stoch-trans}) and multiply both sides by $\Gamma^\prime(f)$
\begin{equation}
  \partial_tf  - \partial_i(\sigma_i\sigma_j)\partial_j f - \frac{1}{2}\sigma_i\sigma_j\partial_i\partial_j f + \sigma\cdot\nabla f \dot{W} + \Div \sigma f\dot{W} = \frac{1}{2}\partial_i\partial_j(\sigma_i\sigma_j) f,
\end{equation}
Using It\^{o}'s formula, we can write the time derivative as
\begin{equation}
 \Gamma^\prime(f)\partial_tf = \partial_t\Gamma(f)- \frac{1}{2}\Gamma^{\prime\prime}(f)(\Div(\sigma f))^2
\end{equation}
and the second order derivative becomes
\begin{equation}
  \Gamma^\prime(f) \frac{1}{2}\sigma_i\sigma_j\partial_i\partial_jf = \frac{1}{2}\sigma_i\sigma_j\partial_i\partial_j\Gamma(f) + \frac{1}{2}\Gamma^{\prime\prime}(f)(\sigma\cdot\nabla f)^2.
\end{equation}
We readily conclude that $\Gamma(f)$ solves
\begin{equation}
\begin{aligned}
  &\partial_t\Gamma(f) - \partial_i(\sigma_i\sigma_j)\partial_j \Gamma(f) - \frac{1}{2}\sigma_i\sigma_j\partial_i\partial_j\Gamma(f) + \sigma\cdot \nabla\Gamma(f)\dot{W} + f\Gamma^\prime(f)\Div\sigma\dot{W}\\
 &\hspace{1in}= f\Gamma^\prime(f)\frac{1}{2}\partial_i\partial_j(\sigma_i\sigma_j)  + \frac{1}{2}\Gamma^{\prime\prime}(f)\left(\Div(\sigma f)^2 - (\sigma\cdot\nabla f)^2\right).
\end{aligned}
\end{equation}
Writing the left-hand side above back in divergence form gives
\begin{equation}\label{eq:renormalized-non-simplified}
  \begin{aligned}
    &\partial_t\Gamma(f)- \frac{1}{2}\partial_i\partial_j(\sigma_i\sigma_j\Gamma(f)) + \Div(\sigma\Gamma(f))\dot{W} + G(f)\Div\sigma\dot{W}\\
    &\hspace{1in}= \frac{1}{2}G(f)\partial_i\partial_j(\sigma_i\sigma_j)  + \frac{1}{2}\Gamma^{\prime\prime}(f)\left(\Div(\sigma f)^2 - (\sigma\cdot\nabla f)^2\right),
\end{aligned}
\end{equation}
where $G(z) = z\Gamma^\prime(z) - \Gamma(z)$. We may now utilize some cancellation on difference of squares term on the right-hand side of (\ref{eq:renormalized-non-simplified}). We find
\begin{equation}
  \begin{aligned}
    \Gamma^{\prime\prime}(f)\left(\Div(\sigma f)^2 - (\sigma\cdot\nabla f)^2\right) &= \Gamma^{\prime\prime}(f)\Div \sigma f\left(\Div(\sigma f) + \sigma\cdot\nabla f\right)\\
    &= G^\prime(f)f(\Div\sigma)^2 + 2G^\prime(f)(\Div\sigma)\sigma\cdot\nabla f\\
    &= G^\prime(f)f(\Div\sigma)^2 + 2(\Div\sigma)\sigma\cdot\nabla G(f)
  \end{aligned}
\end{equation}
where it is worth remarking that $G^\prime(z) = z\Gamma^{\prime\prime}(z)$. Some additional cancellations will take place upon writing the term involving $\nabla G(f)$ above in divergence form
\begin{equation}
    (\Div\sigma)\sigma\cdot\nabla G(f) = \Div\big(\sigma (\Div\sigma)G(f)\big) - \Div(\sigma(\Div \sigma))G(f)
  \end{equation}
Substituting these above expressions into the right-hand-side of equation (\ref{eq:renormalized-non-simplified}) gives
\begin{equation}\label{eq:renormalized-non-simplified-1}
  \begin{aligned}
    &\partial_t\Gamma(f)- \frac{1}{2}\partial_i\partial_j(\sigma_i\sigma_j\Gamma(f)) + \Div(\sigma\Gamma(f))\dot{W} + G(f)\Div\sigma\dot{W}\\
    &\hspace{.5in}= \Div\big(\sigma (\Div\sigma)G(f)\big) + \frac{1}{2}G^\prime(f)f(\Div\sigma)^2\\
    &\hspace{1in}+ G(f)\left(\frac{1}{2}\partial_i\partial_j(\sigma_i\sigma_j)  - \Div(\sigma\Div\sigma)\right)
\end{aligned}
\end{equation}
Finally, the last term on the right-hand side above can be simplified using the relation
\begin{equation}
 \frac{1}{2}\partial_i\partial_j(\sigma_i\sigma_j) - \Div(\sigma\Div\sigma) = \frac{1}{2}\partial_i\sigma_j\partial_j\sigma_i- \frac{1}{2}(\Div\sigma)^2.
\end{equation}
Grouping like terms and recalling $H(z) = zG^\prime(z) - G(z)$ gives the final renormalized form.

\subsection{Notation and Definitions}

In what follows, we will simplify notation. We will fix an arbitrary final time $T >0$ throughout the paper. We will also find it convenient to denote
\begin{equation}
  L^q_t(L^p) = L^q([0,T];L^p(\R^n))
\end{equation}
the space of functions $f:[0,T]\times\R^n \to \R$ satisfying
\begin{equation}
  \|f\|_{L^q_t(L^p)} := \left(\int_0^T\left(\int_{\R^n} |f(t,x)|^p\dx\right)^{q/p}\right)^{1/q} <\infty,
\end{equation}
extended in the usual way when $q$ or $p = \infty$. More generally, for any Banach space $B$ we will write
\begin{equation}
 L^q_t(B) = L^q([0,T];B)
\end{equation}
the Lebesgue Bochner space of strongly measurable functions $f:[0,T] \to B$ such that
\begin{equation}
  \|f\|_{L^q_t(B)} := \left(\int_0^T \|f(t)\|_B^q\dt\right)^{1/q} < \infty.
\end{equation}

For a given $p \in [1,\infty]$ we denote by $p^{\prime} \in [1,\infty]$ it's H\"{o}lder conjugate satisfying
\begin{equation}
  \frac{1}{p} + \frac{1}{p^\prime} = 1.
\end{equation}
We will also denote $\langle\,,\,\rangle$ the pairing between functions on $\R^n$, defined by
\begin{equation}
  \langle f, g\rangle := \int_{\R^n} f\,g\,\dx.
\end{equation}
It follows the $\langle f, g\rangle < \infty$ if $f \in L^p$ and $g\in L^{p^\prime}$.

First, we introduce the precise definition of weak $L^p$ solution to (\ref{eq:stoch-transport})

\begin{definition}
  Let $p \in [2,\infty]$ and
  \begin{equation}
    b\in L^1_t(L^{p^\prime}_\loc), \quad \sigma^k \in L^2_t(L^{2p^{\prime}}_\loc).
  \end{equation}
  A weak $L^p$ solution to (\ref{eq:stoch-transport}) is a measurable mapping $f:\Omega\times[0,T]\times\R^n \to \R$ such that, $\P$ almost surely $f(t,x)$ belongs to $L^\infty_t(L^p_{\loc})$ and for each $\varphi\in C^\infty_c(\R^n)$, the process $t\mapsto \langle f(t),\varphi\rangle$ has a modification that is a continuous $\mathcal{F}_t$-semimartingale and satisfies $\P$ almost surely
  \begin{equation}\label{eq:time-integrated-weak-form}
    \begin{aligned}
      \langle f(t), \varphi\rangle &= \langle f_0,\varphi\rangle + \int_0^t\langle f(s),b(s)\cdot\nabla\varphi\rangle \ds +\frac{1}{2}\int_0^t\langle\sigma_i^k(s)\sigma^k_j(s)f(s),\partial_i\partial_j\varphi\rangle\ds \\
      &\hspace{1in}+ \int_0^t\langle f(s),\sigma^k(s)\cdot\nabla\varphi\rangle \dee W^k(s).
    \end{aligned}
  \end{equation}
\end{definition}

In general we would like to make sense of the renormalized form of a weak $L^p$ solution. Let $\Gamma(z)$ be a smooth function bounded function, and $f$ be a weak $L^p$ solution to (\ref{eq:stoch-transport}), following a similar derivation to that in Section \ref{sec:Renormalized-form}, the renormalized for of (\ref{eq:stoch-transport}) can be shown to be given by
\begin{equation}\label{eq:renormalized-form-1}
\begin{aligned}
  &\partial_t\Gamma(f) +\Div(b\,\Gamma(f)) - \frac{1}{2}\partial_i\partial_j(\sigma^k_i\sigma_j^k\Gamma(f)) + \Div(\sigma^k\Gamma(f))\dot{W}^k+ G(f)\Div\sigma^k \dot{W}^k\\
  &\hspace{.5in}+ G(f)\Div b - \Div\big(\sigma^k(\Div\sigma^k) G(f)\big) =  \frac{1}{2}G(f)\partial_i\sigma_j^k\partial_j\sigma_i^k + \frac{1}{2}H(f)(\Div\sigma^k)^2 , 
\end{aligned}
\end{equation}

It's important to note that the above renormalized equation (\ref{eq:renormalized-form-1}) is in divergence form, so that in distribution, this equation makes sense without any regularity requirements on $\Gamma(f)$ and no more than one derivative on $\sigma^k$ and $b$.

\begin{definition}
  Let $f$ be a weak $L^p$ solution to (\ref{eq:stoch-transport}) for $p\in[2,\infty]$. and suppose that
  \begin{equation}
    \Div b \in L^1_t(L^1_\loc), \quad \Div \sigma \in L^2_t(L^2_{\loc}),\quad \partial_i\sigma^k_j\partial_j\sigma^k_i \in L^1_t(L^1_\loc).
  \end{equation}
  We say that $f$ is renormalizable if for each $\Gamma\in C^2_b(\R^n)$, and each $\varphi \in C^\infty_c(\R^n)$, the process $t\mapsto \langle\Gamma(f(t)),\varphi\rangle$ has a modification that is a continuous $\mathcal{F}_t$-semimartingale and for each $t\in[0,T]$ satisfies $\P$ almost-surely
  \begin{equation}\label{eq:Renorm-weak-form}
    \begin{aligned}
      &\langle \Gamma(f(t)),\varphi\rangle = \langle \Gamma(f(0)),\varphi\rangle + \int_0^t \langle \Gamma(f(s)), b\cdot\nabla\varphi\rangle \ds + \frac{1}{2}\int_0^t\langle \Gamma(f(s))\sigma^k_i\sigma^k_j,\partial_i\partial_j\varphi\rangle\ds\\
      &\hspace{1in}+ \int_0^t \langle \Gamma(f(s)),\sigma^k\cdot\nabla\varphi \rangle\dee W^k(s) - \int_0^t \langle G(f(s))\Div \sigma^k, \varphi\rangle \dee W^k(s)\\
      &\hspace{1in}- \int_0^t\langle G(f(s))\Div b,\varphi\rangle\ds - \int_0^t \langle G(f(s))\Div\sigma^k,\sigma^k \cdot\nabla\varphi\rangle\ds\\
      &\hspace{1in}+ \frac{1}{2}\int_0^t\langle G(f(s))\partial_i\sigma_j^k\partial_j\sigma_i^k,\varphi\rangle\ds + \frac{1}{2}\int_0^t\langle H(f(s))(\Div\sigma^k)^2,\varphi\rangle \ds.
      \end{aligned}
    \end{equation}
\end{definition}

\section{Renormalization for Rough Diffusion}\label{sec:renorm-rough-diff}
In this section we focus on the renormalizability of drift-less stochastic continuity equation with only one diffusion coefficient
\begin{equation}\label{eq:driftless-cont-eq}
  \partial_t f - \frac{1}{2}\partial_i\partial_j(\sigma_i\sigma_jf) + \Div(\sigma f)\dot{W}= 0,
\end{equation}

Our goal will be establish conditions on $\{\sigma\}$ under which weak $L^p$ solutions to (\ref{eq:driftless-cont-eq}) are renormalized solutions.

In what follows, we will find it useful to introduce the differential operators
\begin{equation}
  \begin{aligned}
    \Lgen_\sigma\varphi &:= \frac{1}{2}\sigma_i\sigma_j\partial_i\partial_j\varphi,\quad &\nabla_{\sigma}\varphi &:= \sigma\cdot\nabla\phi\\
    \Lgen^*_{\sigma}\phi &:= \frac{1}{2}\partial_i\partial_j(\sigma_i\sigma_j\varphi),\quad&  \Div_{\sigma}\phi &:= \Div(\sigma\phi).
  \end{aligned}
\end{equation}
With this notation, the stochastic continuity equation (\ref{eq:driftless-cont-eq}) takes the form
\begin{equation}
  \partial_t f - \Lgen^*_\sigma f + \Div_{\sigma} f \dot{W} = 0. 
\end{equation}
and the renormalized form becomes,
\begin{equation}\label{eq:renormalized-form-simple}
  \begin{aligned}
    &\partial_t\Gamma(f) - \Lgen^*_\sigma\Gamma(f) + \Div_\sigma \Gamma(f)\dot{W} + \Div\sigma\, G(f) \dot{W}\\
    &\hspace{1in}= \Div_\sigma(\Div\sigma\, G(f)) + \tfrac{1}{2}G(f)\partial_i\sigma_j\partial_j\sigma_i + \tfrac{1}{2}H(f)(\Div\sigma)^2.
    \end{aligned}
\end{equation}

The main result of this section is the following:
\begin{thm}\label{thm:main-renorm-result}
  Let $f$ be a weak $L^p$ solution, $p\in[2,\infty]$ to (\ref{eq:driftless-cont-eq}), and suppose that $\sigma$ satisfies 
  \begin{equation}
\sigma \in L^2_t(W^{1,\gamma}_\loc), \quad \gamma = \frac{2p}{p-2},
\end{equation}
then $f$ is renormalizable.
\end{thm}

\subsection{Regularization Lemmas}
As in the deterministic theory, commutators of vector field operations with smoothing play an important role in the renormalization theory. Indeed identifying the correct commutators is crucial for simplifying certain remainders in an efficient manner.

We start by considering $\eta: \R^n \to \R$ a smooth, positive, symmetric function with support in the ball of radius $1$ and with unit integral. For each $\ep>0$ we denote by $\eta_\ep$ the rescaled function (mollifier) by
\begin{equation}
  \eta_\ep(x) = \ep^{-n}\eta(\ep^{-1}x).
\end{equation}
For any function $F:\R^n \to \R$, we denote it's convolution with $\eta_\ep$ (or mollifification) by
\begin{equation}
  f_\ep = (f)_\ep = \eta_\ep \star f.
\end{equation}
Define the following quantities
\begin{equation}
  \begin{aligned}
    T_{\sigma,\ep}(f) &:= \nabla_\sigma f_\ep(x) - (\Div_\sigma f)_\ep(x)\\
    &= \int_{\R^n}\nabla\eta_\ep(x - y)\cdot(\sigma(x) - \sigma(y))f(y)\,\dy.
  \end{aligned}
\end{equation}
and
\begin{equation}
  \begin{aligned}
   S_{\sigma,\ep}(f)(x) &:= \Lgen_\sigma f_\ep(x) - \nabla_\sigma(\Div_\sigma f)_\ep(x) + (\Lgen_\sigma^*f)_\ep(x)\\
   &= \frac{1}{2}\int_{\R^n} \nabla^2\eta_\ep(x-y):(\sigma(x) - \sigma(y))^{\tensor 2}f(y)\,\dy.
  \end{aligned}
\end{equation}

Note that the quantity $T_{\sigma,\ep}(f)$ is also studied in \cite{Di_Perna1989-pu} and differs from the usual `commutator' by a divergence term. As a consequence, instead of $T_{\sigma,\ep}(f)$ vanishing as $\ep\to 0$ it will converge precisely to the divergence term that is excluded. Likewise, the quantity $S_{\sigma,\ep}(f)$, while (to the authors knowledge) not consider in the literature, has a similar structure with a non-trivial limit as $\ep\to 0$.

\begin{lem}\label{lem:commutator-Lemma}
  Let $f\in L^p_{\loc}$ and $\sigma\in W^{1,q}_{\loc}$, for $p,q \in [1,\infty]$. Then as $\ep\to 0$
  \begin{equation}
    T_{\sigma,\ep}(f) \to (\Div \sigma) f\quad \text{in}\quad L^r_{\loc},\quad\text{for}\quad  \tfrac{1}{r} = \tfrac{1}{q} + \tfrac{1}{p}
  \end{equation}
and 
\begin{equation}
   S_{\sigma,\ep}(f)  \to \frac{1}{2}(\partial_i\sigma_j\partial_j\sigma_i + (\Div \sigma)^2)f \quad\text{in}\quad L^r_{\loc}, \quad\text{for}\quad \tfrac{1}{r} = \tfrac{2}{q} + \tfrac{1}{p}.
\end{equation}
Moreover for any compact $K\subseteq \R^n$
\begin{equation}\label{eq:T-bound}
  \|T_{\sigma,\ep}(f)\|_{L^r(K)}\leq \|\nabla\sigma\|_{L^q(K)}\|f\|_{L^p(K)},\quad\text{for}\quad  \tfrac{1}{r} = \tfrac{1}{q} + \tfrac{1}{p}
\end{equation}
\begin{equation}\label{eq:S-bound}
  \|S_{\sigma,\ep}(f)\|_{L^r(K)} \leq \|\nabla \sigma\|_{L^q(K)}^2 \|f\|_{L^p(K)}, \quad\text{for}\quad \tfrac{1}{r} = \tfrac{2}{q} + \tfrac{1}{p}.
\end{equation}
\end{lem}
\begin{proof}
  Both bounds (\ref{eq:T-bound}) and (\ref{eq:S-bound}) follow readily from writing
 \begin{equation}
   \sigma(x) - \sigma(y) = \int_{0}^1 \nabla\sigma(y + \lambda(x-y))\cdot(x-y)\dee \lambda
 \end{equation}
 and applying generalized Minkowski and Holder's inequality to $T_{\sigma,\ep}(f)$ and $S_{\sigma,\ep}(f)$ along with the fact that $x\tensor \nabla \eta_\ep(x)$ and $x^{\tensor 2}\tensor \nabla^2 \eta_\ep(x)$ are uniformly bounded in $L^1_{\loc}$
 
  To show the limits, we study $T_{\sigma,\ep}(f)$ first. Define for each $x,w\in\R^n$ the quantity
  \begin{equation}
    \begin{aligned}
      R_w(x) &:= \sigma(x) - \sigma(x-w) - \nabla\sigma(x)\cdot w\\
      &= \int_0^1 (\nabla\sigma(x + (\lambda-1) w) - \nabla\sigma(x))\cdot w\, \dee\lambda,
    \end{aligned}
  \end{equation}
  so that we can write 
  \begin{equation}
     T_{\sigma,\ep}(f)(x) = \int_{\R^n}\nabla\eta_\ep(x-y)\cdot(\nabla\sigma(x)\cdot(x-y))f(y)\dy + \int_{\R^n} \nabla \eta_\ep(y)R_y(x) f(x-y)\dy
  \end{equation}
Since $\sigma\in W^{1,q}_{\loc}$ we have that if $|w|< \ep$, then for any compact $K\subseteq \R^n$
  \begin{equation}
    \|R_w(x)\|_{L^q(K)} \leq \ep \sup_{|y|<\ep}\|\delta_y \nabla \sigma\|_{L^q_{x}(K)},
  \end{equation}
  where $\delta_yh(x)= h(x+y) - h(x)$ denotes the difference of for some function $h$ and it's translation by $y$. Using the above bound, and the fact that $\ep \|\nabla\eta_\ep\|_{L^1}$ is uniformly bounded in $\ep$, we find for 
\begin{equation}
\begin{aligned}
  \left\|\int_{\R^n} \nabla \eta_\ep(y)R_y(x) f(x-y)\dy\right\|_{L^r(K)} &\leq \left(\int_{\R^n} \nabla \eta_\ep(y) \,\|R_y\|_{L^q(K)}\dy\right)\|f\|_{L^p(K)}\\
&\leq \ep\|\nabla\eta_\ep\|_{L^1}\sup_{|y|<\ep}\|\delta_{y}\nabla\sigma\|_{L^q(K)}\|f\|_{L^p(K)}\\
& \to 0 \quad \text{as}\quad \ep \to 0.
\end{aligned}
\end{equation}
for any compact set $K\subset \R^n$. Indeed this implies that for each $x\in \R^n$, and $r$ satisfying $\tfrac{1}{r} = \tfrac{1}{q} + \tfrac{1}{p}$,
  \begin{equation}
      T_{\sigma,\ep}(f)(x) = \nabla\sigma(x):(M_\ep\star f) + o(1)_{L^r_{\loc}}
  \end{equation}
  where $M_\ep(x) = x\tensor\nabla\eta_\ep(x)$. Furthermore, using the fact that each component of $M_\ep(x) = \ep^{-n}M(\ep^{-1}x)$ is a symmetric approximation of a delta function, we can use the standard properties of mollifiers to find
  \begin{equation}
    M_\ep \star f \to \Big(\int_{\R^n} x\tensor\nabla\eta(x) \dx\Big)\, f \quad\text{a.e. on }\quad [0,T]\times \R^n.
  \end{equation}
  Integration by parts, and the properties of $\eta$ give the identity
  \begin{equation}
    \int_{\R^n} x\tensor\nabla \eta(x)\dx = I.
  \end{equation}
  Since $\nabla\sigma \in L^q_{\loc}$ and $M_\ep \star f$ is uniformly bounded in $L^p_{\loc}$, as $\ep \to 0$, the following convergence holds in $L^r_{\loc}$
  \begin{equation}
    T_{\sigma,\ep}(f)(f) \to (\nabla \sigma : I) f = (\Div \sigma) f.
  \end{equation}
  Note that in the case $p = \infty$ we must resort to the product limit Lemma \ref{lem:prod-lim}.
  
  Next we study $S_{\sigma,\ep}(f)$. A similar argument to the one just given for $T_{\sigma,\ep}(f)$ also implies that since $\sigma \in W^{1,q}_{\loc}$, then for each $x\in\R^n$ and $r$ such that $\tfrac{1}{r} = \tfrac{2}{q} + \tfrac{1}{p}$,
  \begin{equation}
    \begin{aligned}
      S_{\sigma,\ep}(f)(x) &= \frac{1}{2}\int_{\R^n} \nabla^2\eta_\ep(x-y):(\nabla\sigma(x)\cdot(x-y))^{\tensor 2}f(y)\dy + o(1)_{L^r_{\loc}}\\
&= \frac{1}{2}\partial_i\sigma_k(x)\partial_j\sigma_\ell(x)\big(M_{ij,\ep}^{kl}\star f)(x) + o(1)_{L^r_{\loc}},
    \end{aligned}
  \end{equation}
  where $M_{ij,\ep}^{k\ell}(x) = x_ix_j\partial_k\partial_\ell \eta_\ep(x)$. Furthermore, since each $M_{ij,\ep}^{k\ell}(x)$ is a symmetric approximation of a delta function we have
 \begin{equation}
   M_{ij,\ep}^{k\ell}\star f \to \left(\int_{\R^n}x_ix_j\partial_k\partial_\ell \eta\dx\right) f \quad\text{a.e. on }\quad [0,T]\times \R^n.
 \end{equation}
 Using the identity,
 \begin{equation}
   \int_{\R^n}x_ix_j\partial_k\partial_\ell \eta(x)\dx = \delta_{ij}\delta_{k\ell} + \delta_{ik}\delta_{j\ell},
 \end{equation}
 and the facts that $\nabla \sigma \in L^q_{\loc}$, and $M_{ij,\ep}^{k\ell}\star f$ is uniformly bounded in $L^p_{\loc}$, the following convergence holds in $L^r_{\loc}$ as $\ep \to 0$
 \begin{equation}
   S_{\sigma,\ep}(f) \to \frac{1}{2}(\delta_{ij}\delta_{k\ell} + \delta_{ik}\delta_{j\ell}) \partial_i\sigma_k\partial_j\sigma_\ell f = \frac{1}{2}(\partial_i\sigma_j\partial_j\sigma_i + (\Div\sigma)^2)f.
 \end{equation}

\end{proof}

\subsection{Proof of Theorem \ref{thm:main-renorm-result}}
\begin{proof}
\noindent {\bf Step 1 (Regularization)}:

\vspace{1em}

The first step is to mollify the stochastic continuity equation. For a weak $L^p$ solution this can be done by choosing for each $\ep >0$ and $x$ a text function $\eta_\ep(x-\cdot)$. This gives 
  \begin{equation}
    f_\ep(t) = f_\ep(0) + \int_0^t(\Lgen_\sigma^*f(s))_\ep\ds - \int_0^t(\Div_\sigma f(s))_\ep\dee W(s).
  \end{equation}
Then, using It\^{o}s formula applied to $\Gamma(f_\ep)$, we have
\begin{equation}
  \begin{aligned}
    &\Gamma(f_\ep(t)) = \Gamma(f_\ep(0)) + \int_0^t\Gamma^\prime(f_\ep(s))(\Lgen_\sigma^* f(s))_\ep\ds -  \int_0^t\Gamma^\prime(f_\ep(s))(\Div_\sigma f)_\ep \dee W(s)\\
    &\hspace{1in}+ \frac{1}{2}\int_0^t\Gamma^{\prime\prime}(f_\ep(s))(\Div_\sigma f)_\ep^2\ds
    \end{aligned}
  \end{equation}
We can then write the equation above as a stochastic continuity equation for $\Gamma(f)$ plus some remainders. Specifically, we have
\begin{equation}\label{eq:stochastic-renorm-remainder}
  \begin{aligned}
    &\Gamma(f_\ep(t)) = \Gamma(f_\ep(0)) + \int_0^t\Lgen_\sigma^*\Gamma(f_\ep(s))\ds -  \int_0^t\Div_\sigma \Gamma(f_\ep(s)) \dee W(s)\\
    &\hspace{1in}+ \int_0^t R^1_\ep(f(s))\dee W(s) + \int_0^t R_\ep^2(f(s))\ds.
  \end{aligned}
\end{equation}
where the remainders $R^1_\ep(f)$ and $R^2_\ep(f)$ are given by
\begin{equation}
\begin{aligned}
  &R^1_\ep(f) = \Div_\sigma\Gamma(f_\ep) - \Gamma^\prime(f_\ep)(\Div_\sigma f)_\ep\\
  &R^2_\ep(f) = \Gamma^\prime(f_\ep)(\Lgen_\sigma^* f)_\ep - \Lgen^*_\sigma\Gamma(f_\ep) + \frac{1}{2}\Gamma^{\prime\prime}(f_\ep)(\Div_\sigma f)_\ep^2
\end{aligned}
\end{equation}

\noindent{\bf Step 2 (Cancellation Identity):}

In order to complete the proof we need to show that as $\ep \to 0$, the remainders $R_\ep^1(f)$ and $R^2_\ep(f)$ converge to the correct terms on the right-hand side of (\ref{eq:renormalized-form-simple}). To show this, we will need to write $R_\ep^1(f)$ and $R_\ep^2(f)$ in terms of $T_{\sigma,\ep}(f)$ and $S_{\sigma,\ep}(f)$ and use Lemma \ref{lem:commutator-Lemma}. In fact, we will show the following identities
  \begin{equation}\label{eq:R1-id}
    R^1_\ep(f) = \Gamma^\prime(f_\ep)T_{\sigma,\ep}(f) - \Div\sigma \Gamma(f_\ep)
  \end{equation}
  and 
  \begin{equation}\label{eq:R2-id}
    \begin{aligned}
      &R^2_\ep(f)= \Gamma^\prime(f_\ep)S_{\sigma,\ep}(f) - \frac{1}{2}\Gamma(f_\ep)\partial_i\sigma_j\partial_j\sigma_i -\frac{1}{2}\Gamma^\prime(f_\ep)(\Div\sigma) T_{\sigma,\ep}(f)\\
      &\hspace{1in}+\frac{1}{2}\Gamma^{\prime\prime}(f_\ep)(T_{\sigma,\ep}(f))^2+ \Div_\sigma\big(R^1_\ep(f)\big) - \frac{1}{2}(\Div \sigma)R^1_\ep(f).
  \end{aligned}
  \end{equation}

  It is important to remark that this computation involves quantities, like $\partial_i\partial_j(\sigma_i\sigma_j)$ and $\nabla_\sigma\Div \sigma$ which are not well-defined functions given the regularity assumptions on $\sigma$. However, as they are well-defined distributions and always multiplied by a smooth function (for $\ep >0$) the computations are well defined in the sense distribution for $\sigma \in L^2_t(W^{1,\gamma}_{\loc})$.

The proof of the identity (\ref{eq:R1-id}) for $R^1_\ep(f)$ is obvious given the definition of $T_{\sigma,\ep}(f)$. We focus on identity (\ref{eq:R2-id}) $R^2_\ep(f)$ and begin by expanding the term for $\Lgen^*_\sigma\Gamma(f)$,
\begin{equation}
  \begin{aligned}
    \Lgen^*_\sigma\Gamma(f_\ep) = \frac{1}{2}\partial_i\partial_j(\sigma_i\sigma_j)\Gamma(f_\ep) + \Gamma^\prime(f_\ep) \Div_\sigma \sigma \cdot \nabla f_\ep + \Gamma^{\prime\prime}(f_\ep) (\nabla_\sigma f_\ep)^2 + \Gamma^{\prime}(f_\ep)\Lgen_\sigma f_{\ep}
  \end{aligned}
\end{equation}
So that $R^2_\ep$ becomes
\begin{equation}
  \begin{aligned}
    &R^2_\ep(f) = -\frac{1}{2}\partial_i\partial_j(\sigma_i\sigma_j) \Gamma(f_\ep) - \Gamma^\prime(f_\ep) \Div_\sigma\sigma \cdot \nabla f_\ep  - \Gamma^{\prime}(f_\ep)(\Lgen_\sigma f_{\ep} -(\Lgen^*_\sigma f)_\ep)\\
    &\hspace{1in}+ \frac{1}{2}\Gamma^{\prime\prime}(f_\ep)((\Div_\sigma f)_\ep^2- (\nabla_\sigma f_\ep)^2)
  \end{aligned}
\end{equation}
We can write several expressions in terms of commutators
\begin{equation}
  \begin{aligned}
    \Lgen_\sigma f_\ep - (\Lgen^*_\sigma f)_\ep &= -S_{\sigma,\ep}(f) + 2\Lgen_\sigma f_\ep - \nabla_\sigma(\Div_\sigma f)_\ep\\
    &= -S_{\sigma,\ep}(f) + 2\Lgen_\sigma f_\ep - \nabla_\sigma\nabla_\sigma f_\ep + \nabla_\sigma T_{\sigma,\ep}(f)\\
    &= -S_{\sigma,\ep}(f) - \nabla_\sigma\sigma\cdot \nabla f_\ep + \nabla_\sigma T_{\sigma,\ep}(f)
    \end{aligned}
  \end{equation}
  and
  \begin{equation}
    \begin{aligned}
      (\Div_\sigma f)_\ep^2- (\nabla_\sigma f_\ep)^2 &= ( \nabla_\sigma f_\ep - T_{\sigma,\ep}(f))^2 -(\nabla_\sigma f_\ep)^2\\
      &= -2(\nabla_\sigma f_\ep) T_{\sigma,\ep}(f) + T_{\sigma,\ep}(f)^2
    \end{aligned}
  \end{equation}
  Therefore we have
  \begin{equation}
    \begin{aligned}
      &\frac{1}{2}\Gamma^{\prime\prime}(f_\ep)\big((\Div_\sigma f)_\ep^2- (\nabla_\sigma f_\ep)^2\big) - \Gamma^{\prime}(f_\ep)(\Lgen_\sigma f_{\ep} -(\Lgen^*_\sigma f)_\ep)\\
      &\hspace{.5in}= \Gamma^{\prime}(f_\ep)S_{\sigma,\ep}(f) +\frac{1}{2}\Gamma^{\prime\prime}(f_\ep)(T_{\sigma,\ep}(f))^2 - \Gamma^{\prime\prime}(f_\ep)(\nabla_\sigma f_\ep) T_{\sigma,\ep}(f) - \Gamma^\prime(f_\ep)\nabla_\sigma T_{\sigma,\ep}(f)\\
      &\hspace{1in}+  \Gamma^\prime(f_\ep)\nabla_\sigma\sigma\cdot \nabla f_\ep\\
      &\hspace{.5in}= \Gamma^\prime(f_\ep)S_{\sigma,\ep}(f) +\frac{1}{2}\Gamma^{\prime\prime}(f_\ep)(T_{\sigma,\ep}(f))^2 + \nabla_\sigma\big(\Gamma^\prime(f) T_{\sigma,\ep}(f)\big)+  \Gamma^\prime(f_\ep)\nabla_\sigma\sigma\cdot \nabla f_\ep
    \end{aligned}
  \end{equation}
  The remainder becomes
  \begin{equation}
    \begin{aligned}
      &R^2_\ep(f) = \Gamma^\prime(f_\ep)S_{\sigma,\ep}(f) +\frac{1}{2}\Gamma^{\prime\prime}(f_\ep)(T_{\sigma,\ep}(f))^2 + \nabla_\sigma\big(\Gamma^\prime(f_\ep) T_{\sigma,\ep}(f)\big)+  \Gamma^\prime(f_\ep)\nabla_\sigma\sigma\cdot \nabla f_\ep\\
      &\hspace{1in}-\frac{1}{2}\partial_i\partial_j(\sigma^k_i\sigma_j^k) \Gamma(f_\ep) - \Gamma^\prime(f_\ep) \Div_\sigma\sigma \cdot \nabla f_\ep
  \end{aligned}
\end{equation}
Next we write
\begin{equation}
  \begin{aligned}
    \Gamma^\prime(f_\ep)\nabla_\sigma\sigma\cdot \nabla f_\ep - \Gamma^\prime(f_\ep) \Div_\sigma a \cdot \nabla f_\ep &= - \Div{\sigma}\nabla_\sigma \Gamma(f_\ep)\\
    &= -\nabla_\sigma(\Div\sigma \Gamma(f_\ep)) + \nabla_\sigma\Div\sigma \Gamma(f_\ep),
\end{aligned}
\end{equation}
and use the fact that $\frac{1}{2}\partial_i\partial_j(\sigma_i\sigma_j) = \tfrac{1}{2}\partial_i\sigma_j\partial_j\sigma_i + \tfrac{1}{2}(\Div\sigma)^2 + \nabla_\sigma\Div \sigma$ to simplify the remainder to
  \begin{equation}
    \begin{aligned}
 &R^2_\ep(f) = \Gamma^\prime(f_\ep)S_{\sigma,\ep}(f) +\frac{1}{2}\Gamma^{\prime\prime}(f_\ep)(T_{\sigma,\ep}(f))^2 + \nabla_\sigma\big(\Gamma^\prime(f_\ep) T_{\sigma,\ep}(f)\big) -\nabla_\sigma(\Div\sigma \Gamma(f_\ep))\\
      &\hspace{1in} -\frac{1}{2}\Gamma(f_\ep)\partial_i\sigma_j\partial_j\sigma_i - \frac{1}{2}\Gamma(f_\ep)(\Div\sigma)^2.
  \end{aligned}
\end{equation}
Identity (\ref{eq:R2-id}) now follows by writing $\nabla_\sigma = \Div_\sigma - \Div\sigma$ in two of the terms above, grouping like terms, and recognizing the appearance of $R^1_\ep(f)$.

\vspace{1em}

\noindent{\bf Step 3 (Passing to the limit):}

\vspace{1em}

We need to pass the limit as $\ep \to 0$ in the weak, time integrated form of (\ref{eq:stochastic-renorm-remainder}) which for each $\varphi \in C^\infty_c(\R^n)$ becomes
\begin{equation}\label{eq:mollified-weak-form}
  \begin{aligned}
    &\langle \Gamma(f_\ep(t)),\varphi\rangle  = \langle\Gamma(f_\ep(0)),\varphi\rangle + \int_0^t \langle \Gamma(f_\ep(s)), \Lgen_\sigma\varphi\rangle \ds + \int_0^t \langle \Gamma(f_\ep(s)),\nabla_\sigma \varphi \rangle\dee W(s)\\
    &\hspace{1in}+ \int_0^t \langle R^1_\ep(f(s)), \varphi\rangle \dee W(s) + \int_0^t \langle R^2_\ep(f(s)),\varphi\rangle \ds .
  \end{aligned}
\end{equation}

The standard properties of mollifiers imply that for $\P\tensor \dt$ almost every $(\omega,t)\in \Omega\times [0,T]$, $f_\ep(t) \to f(t)$ in $L^p_{\loc}$ and $f_\ep \to f$ pointwise $\P\tensor\dt\tensor\dx$ almost everywhere on $\Omega\times[0,T]\times\R^n$. It is a simple matter to show that this, the boundedness of $\Gamma(z)$ and the integrability conditions on $\sigma$ imply that as $\ep \to 0$,
\begin{equation}
  \begin{aligned}
&\langle \Gamma(f_\ep(0)), \varphi\rangle \to \langle\Gamma(f_0),\varphi\rangle,\\
    &\langle\Gamma(f_\ep),\varphi\rangle \to \langle \Gamma(f),\varphi\rangle,\quad\text{in}\quad L^1(\Omega\times[0,T])\\
    &\langle \Gamma(f_\ep),\Lgen_\sigma\varphi\rangle \to \langle \Gamma(f),\Lgen_\sigma\varphi\rangle,\quad\text{in}\quad L^1(\Omega\times [0,T])
    \end{aligned}
  \end{equation}
while for the term in the stochastic integral
  \begin{equation}
    \langle \Gamma(f_\ep),\nabla_\sigma\varphi\rangle \to\langle \Gamma(f),\nabla_\sigma\varphi\rangle \quad \text{in} \quad L^2(\Omega\times[0,T]).
  \end{equation}
Consequently we may pass the limit as $\ep \to 0$ in the first four terms of equation (\ref{eq:mollified-weak-form}).

 What remain are the terms involving $R_\ep^1(f)$ and $R^2_\ep(f)$. Lemma \ref{lem:commutator-Lemma}, the fact that $\sigma \in L^2_t(W^{1,\gamma}_{\loc})$, $\gamma = 2p/(p-2)$ and the strong convergence properties of $f_\ep \to f$ are more than enough to conclude
 \begin{equation}
   \langle R^1_\ep(f),\varphi\rangle \to \langle G(f)\Div \sigma,\varphi \rangle, \quad \P\tensor \dt\text{ almost everywhere}.
\end{equation}
Moreover, the bound provided in Lemma \ref{lem:commutator-Lemma}, and the fact that $\Gamma(z)$ and $\Gamma^\prime(z)$ are bounded functions give the estimate
\begin{equation}
 \sup_\ep |\langle R^1_\ep(f(t)),\varphi \rangle|^2 \leqc \|\nabla\sigma(t)\|_{L^{\gamma}(\supp\varphi))}^2(1+ \|f\|^2_{L^\infty_t(L^p)})
\end{equation}
where $\supp\varphi$ is the support of $\varphi$. Since the right-hand side of the above inequality is integrable with respect to $\P\tensor \dt$ on $\Omega\times[0,T]$, the dominated convergence theorem implies that
  \begin{equation}
    \langle R^1_\ep(f),\varphi\rangle \to \langle (\Div \sigma) G(f), \varphi\rangle \quad\text{in}\quad L^2(\Omega\times[0,T]),
  \end{equation}
  whereby we may pass the limit in the stochastic integral for the fourth term on the right-hand side of (\ref{eq:stochastic-renorm-remainder}).

  The last term $R^2_\ep(f)$, though complicated, is straightforward and can be treated in a similar manner as $R^1_\ep(f)$. Indeed similar arguments to those above show that
  \begin{equation}
    \begin{aligned}
  &\langle R^1_\ep(f), \sigma\cdot \nabla\varphi\rangle \to \langle (\Div\sigma)G(f), \sigma\cdot \nabla\varphi\rangle \quad \text{in}\quad L^1(\Omega\times[0,T]),\\
  &\frac{1}{2}\langle (\Div\sigma) R^1_{\ep}(f),\varphi\rangle \to \frac{1}{2}\langle G(f) (\Div\sigma)^2, \varphi\rangle \quad \text{in}\quad L^1(\Omega\times[0,T]),\\
  &\frac{1}{2}\langle\Gamma^{\prime\prime}(f_\ep)(T_{\sigma,\ep}(f))^2,\varphi\rangle \to \frac{1}{2}\langle fG^\prime(f)(\Div\sigma)^2,\varphi\rangle \quad \text{in}\quad  L^1(\Omega\times[0,T]).
  \end{aligned}
\end{equation}

The only term left to study in $R^2_\ep(f)$ is the one involving $S_{\sigma,\ep}(f)$. Of course Lemma \ref{lem:commutator-Lemma} applied to $S_{\sigma,\ep}(f)$ combined with the strong convergence of $f_\ep(t) \to f(t)$ in $L^p$ implies that
\begin{equation}
  \langle \Gamma^\prime(f_\ep)S_{\sigma,\ep}(f),\varphi\rangle \to \frac{1}{2}\langle \Gamma^\prime(f_\ep)(\partial_i\sigma_j\partial_j\sigma_i + (\Div\sigma)^2),\varphi\rangle \quad \P\tensor \dt\text{ almost everywhere}
\end{equation}
as well as the bound
\begin{equation}
  \sup_\ep |\langle \Gamma^\prime(f_\ep(t))S_{\sigma,\ep}(f(t)),\varphi\rangle| \leqc \|\nabla\sigma(t)\|^2_{L^\gamma(\supp \varphi)} \|f\|_{L^\infty_t(L^p)}.
\end{equation}
Again applying dominated convergence gives
\begin{equation}
  \begin{aligned}
    &\left\langle \Gamma^\prime(f_\ep)S_{\sigma,\ep}(f) - \tfrac{1}{2}\Gamma^\prime(f_\ep)(\partial_i\sigma_j\partial_j\sigma_i + (\Div\sigma)^2),\varphi\right\rangle\\
    &\hspace{1in}\to \frac{1}{2}\langle G(f)(\partial_i\sigma_j\partial_j\sigma_i + (\Div\sigma)^2),\varphi\rangle\quad \text{in}\quad L^1(\Omega\times[0,T]).
\end{aligned}
\end{equation}

The above limits can be collected and combined to conclude that
\begin{equation}
  \begin{aligned}
    &\langle R^2_\ep(f),\varphi\rangle \to -\langle G(f)\Div\sigma, \nabla_\sigma\varphi\rangle + \frac{1}{2}\langle G(f)\partial_i\sigma_j\partial_j\sigma_i ,\varphi\rangle\\
    &\hspace{1in}+ \frac{1}{2}\langle H(f) (\Div \sigma)^2,\varphi\rangle\quad\text{in}\quad L^1(\Omega\times [0,T]).
    \end{aligned}
  \end{equation}
  All of these convergences, allow us to pass the limit in each term of (\ref{eq:mollified-weak-form}) in $L^1(\Omega\times[0,T])$ and therefore that equation (\ref{eq:Renorm-weak-form}) holds $\P\tensor \dt$ almost everywhere.
  
  \medskip
  
This completes the proof of Theorem \ref{thm:main-renorm-result}
\end{proof}

\begin{remark}
It is useful to note the term $\frac{1}{2}\Gamma^{\prime\prime}(f_\ep)(T_{\sigma,\ep}(f))^2$ in (\ref{eq:R2-id}) is the one that dictate the $L^2_t(W^{1,2p/(p-2)})$ condition on $\sigma$ (as opposed to $L^2_t(W^{1,2p/(p-1)})$ which is sufficient to obtain all the other limits).
\end{remark}

\subsection{Extensions}
Theorem \ref{thm:main-renorm-result} can be easily generalized to equations with a drift $u$ and a family of noise coefficients $\sigma = \{\sigma^k\}$, 
\begin{equation}\label{eq:stochastic-transport-eq-1}
  \partial_t f + \Div(b f)  -\frac{1}{2}\partial_i\partial_j(\sigma_i^k\sigma_j^k f) + \sum_k\Div(\sigma^k f)\dot{W}^k =0,
\end{equation}
as long as $u$ satisfies the usual regularity requirements of the deterministic DiPerna-Lions theory and $\sigma = \{\sigma^k\}$ satisfy the appropriate summability conditions. In this case, the renormalized form looks like
\begin{equation}\label{eq:Full-renormalized-form}
\begin{aligned}
  &\partial_t\Gamma(f) +\Div(b\,\Gamma(f)) - \frac{1}{2}\partial_i\partial_j(\sigma^k_i\sigma_j^k\Gamma(f)) + \Div(\sigma^k\Gamma(f))\dot{W}^k+ G(f)\Div\sigma^k \dot{W}^k\\
  &\hspace{.5in}+ G(f)\Div b - \Div\big(\sigma^k(\Div\sigma^k) G(f)\big) =  \frac{1}{2}G(f)\partial_i\sigma_j^k\partial_j\sigma_i^k + \frac{1}{2}H(f)(\Div\sigma^k)^2.
\end{aligned}
\end{equation}
The corresponding renormalization result is given below:
\begin{thm}
  \label{thm:main-renorm-result-generalized}
Let $f$, be a weak $L^p$ solution to (\ref{eq:stochastic-transport-eq-1}), with $p\in [2,\infty]$ and 
\begin{equation}
  b\in L^1_t(W^{1,p^\prime}_{\loc}),\quad \sigma^k \in L^2_t(W^{1,\gamma}_{\loc}),\quad \gamma = \frac{2p}{p-2},
\end{equation}
then $f$ is renormalizable.
  \end{thm}
  \begin{proof}
    The proof is an easy extension of the Theorem \ref{thm:main-renorm-result}. Regulaizing and renormalizing just as in the proof of Theorem \ref{thm:main-renorm-result}, we see that the drift $b$ introduces another commutator $T_{b,\ep}(f) = b\cdot\nabla f_\ep- (\Div(bf))_\ep$, applying Lemma \ref{lem:commutator-Lemma} with $b$ instead of $\sigma$, we obtain
    \begin{equation}
      \langle \Gamma^\prime(f_\ep)T_{b,\ep}(f)- \Gamma(f_\ep)\Div b,\varphi\rangle \to \langle G(f)\Div b,\varphi \rangle \quad \text{in}\quad L^1(\Omega\times[0,T])
    \end{equation}
as long as $b\in L^1([0,T]; W^{1,p^\prime}_{\loc})$. This limit can be passed along with all of the others as in the proof of Theorem \ref{thm:main-renorm-result}.
\end{proof}

\subsection{Existence}
We now turn to the question of existence and uniqueness of weak $L^p$ solutions to (\ref{eq:stochastic-transport-eq-1}). For simplicity, we will consider the case where there are only finite many coefficients $\{\sigma^k\}$ in order to remove any unsightly summability conditions and simplify the presentation. However, all of the results can easily be extended to infinitely many $\{\sigma^k\}$ with the addition of the appropriate summability conditions. 

The main theorems of this section are:
\begin{thm}[Existence]\label{thm:existence}
  Let $p\in [1,\infty)$ and suppose that
  \begin{equation}
    b\in L^1_t(L^{p^\prime}_\loc), \quad \sigma^k \in L^2_t(L^{2p^\prime}_{\loc}),
  \end{equation}
  as well as
  \begin{equation}
    \Div b \in L^1_t(L^\infty), \quad \Div\sigma^k \in L^2_t(L^\infty), \quad \partial_i\sigma^k_j\partial_j\sigma^k_i \in L^1_t(L^\infty),
  \end{equation}
  Then for each stochastic basis $(\Omega,\mathscr{F}, \P, (\mathscr{F}_t), \{W^k\})$ and $f_0 \in L^p_{\loc}$, there exists a weak $L^p$ solution $f \in L^{2p}(L^\infty_t(L^p))$ to (\ref{eq:stochastic-transport-eq-1}).
\end{thm}

\begin{remark}
  Currently we are unable to prove existence of weak $L^\infty$ solutions to (\ref{eq:stochastic-transport-eq-1}, unless $\Div \sigma^k = 0$. Additionally we are also unable to obtain higher moments in $\Omega$ unless $\Div\sigma^k \in L^\infty_t(L^\infty)$. These subtleties will be addressed in a future work.
\end{remark}

Before we prove Theorem \ref{thm:existence}, we will need a couple of preliminary Lemmas. Fix a cannonical stochastic basis  $(\Omega,\mathscr{F}, (\mathscr{F}_t),\P,\{W^k\})$ and assume that $b$ and $\sigma^k$ are smooth. Let $x\mapsto\Phi_{s,t}(x)$ be the regular stochastic flow associated with the SDE
\begin{equation}\label{eq:ito-diffusion}
    \dee \Phi_{s,t}(x) = b(t,\Phi_{s,t}(x))\dt + \sigma^k(t,\Phi_{s,t}(x))\dee W(t)^k,\quad \Phi_{s,s}(x) = x.
\end{equation}
According to Kunita \cite{Kunita1986-sr}, $x\mapsto\Phi_{s,t}(x)$ is a diffeomorphism and adapted to $\mathscr{F}_{s,t} = \sigma(\{W_t^k - W^k_s\} : 0\leq s\leq t\leq \infty)$. Moreover, its spatial inverse $\Psi_{s,t}(x) = \Phi^{-1}_{s,t}(x)$ is also $(\mathscr{F}_{s,t})$-adapted (pontwise in $x$). We will denote $\Phi_t = \Phi_{0,t}$ and $\Psi_{t} = \Psi_{0,t}$. If $f_0$ smooth, then it follows from Kunita \cite{Kunita1986-sr} that the unique solution to (\ref{eq:stochastic-transport-eq-1}) is given by
\begin{equation}
  f(t,x) = (\Phi_t)_{\#}f_0(x) = f_0(\Psi_t(x))\det{\partial\Psi_t(x)}
\end{equation}
where $(\partial\Psi_t)_{ij} = \partial_j(\Psi_{t})_i$ is the Jacobian matrix. The following Lemma regarding a formula for $\det{\partial\Psi_t(x)}$ will prove crutial for obtaining apriori estimates.
\begin{lem}\label{lem:dpsi-rep}
Assume that $b$ and $\sigma^k$ are smooth and compactly supported, then the following formula for $\det{\partial\Psi_t(x)}$ holds
  \begin{equation}
    \begin{aligned}
      \det{\partial\Psi_t(x)} &= \exp\bigg\{-\int_0^t \Big[\Div b(s,\Psi_{s,t}(x)) - \tfrac{1}{2}(\partial_i\sigma^k_j\partial_j\sigma^k_i)(s,\Psi_{s,t}(x))\Big]\ds\\
        &\hspace{1in}- \int_0^t\Div\sigma^k(s,\Psi_{s,t}(x))\dee W^k(s)\bigg\}.
     \end{aligned}
  \end{equation}
\end{lem}
\begin{proof}
To study $\det{\partial\Psi_t(x)}$ further, we remark that it suffices to study $\det{\partial\Phi_t(x)}$, since we have $\partial\Phi_t(\Psi(x))\partial\Psi_t(x) = I$ and therefore
\begin{equation}
  \det{\partial\Psi_t(x)} = [\det{\partial\Phi_t(\Psi_t(x))}]^{-1}.
\end{equation}
The taking the derivative of both sides of the SDE with respect to the initial data, it is well known that the matrix $\partial\Phi_t$ satisfies
\begin{equation}
  \dee \partial\Phi_{t} = \partial b(t,\Phi_{t})\partial\Phi_{t}\dt +\partial\sigma^k(t,\Phi_{t})\partial\Phi_{t}\dee W^k_t,
\end{equation}
To study the determinant of $\partial\Phi_t$, we use the fact that for any invertible matrix $A$, the Gateaux derivative of $F(A) := \log{(\det{A})}$ in the direction $U$ is
\begin{equation}
  DF(A)\big[U\big] = \Tr\big(UA^{-1}\big)
\end{equation}
while the second order Gateaux derivative in the directions $U,V$ is
\begin{equation}
  D^2F(A)\big[U,V\big] = -\Tr\big(UA^{-1}VA^{-1}\big).
\end{equation}
Applying the classical It\^{o}'s formula to quantity $F(\partial\Phi_t) = \log{(\det\partial\Phi_t(x)})$ (pointwise in $x$), and using the above formulas, we find
\begin{equation}
  \begin{aligned}
    \dee F(\partial\Phi_t) &= DF(\partial\Phi_t)\big[\partial b(t,\Phi_t)\partial\Phi_t\big]\dt + DF(\partial\Phi_t)\big[\partial \sigma^k(t,\Phi_t)\partial\Phi_t\big]\dee W^k_t\\
    &\hspace{1in}+ \tfrac{1}{2}D^2F(\partial\Phi_t)\Big[\partial\sigma^k(t,\Phi_t)\partial\Phi_t,\partial\sigma^k(t,\Phi_t)\partial\Phi_t\Big]\dt\\
    &=\Tr\big(\partial b\big)(t,\Phi_t(x))\dt + \Tr\big(\partial\sigma^k\big)(t,\Phi_t)\dee W^k_t - \frac{1}{2}\Tr\big((\partial\sigma^k)(\partial\sigma^k)\big)(t,\Phi_t)\dt\\
    &= \Div b(t,\Phi_t)\dt  + \Div \sigma^k(t,\Phi_t)\dee W^k_t - \frac{1}{2}(\partial_i\sigma^k_j\partial_j\sigma^k_i)(t,\Phi_t)\dt
    \end{aligned}
  \end{equation}
  Using the fact that $F(\partial\Phi_0) = 0$ concludes the proof.
\end{proof}

We deduce apriori estimates on $f$ for smooth $b$ and $\sigma^k$.
\begin{lem}\label{lem:apriori-est}
  Assume that $f_0, b$ and $\sigma^k$ are smooth and compactly supported and let $f(t,x)$ be the unique classical solution to (\ref{eq:stochastic-transport-eq-1}) For each $p\in [1,\infty)$, we have the following inequality
  \begin{equation}\label{eq:Lp-estimate-existence}
    \E\|f\|_{L^\infty_t(L^p)}^{2p}\leq C_{p,b,\sigma} \|f_0\|_{L^p}^{2p}
  \end{equation}
$C_{p,b,\sigma}$ is a constant that depends on $p$, $b$ and $\sigma^k$ through the norms $\|\Div b\|_{L^1_t(L^\infty)}$, $\|\Div \sigma^k\|_{L^2_t(L^\infty)}$ and $\|\partial_i\sigma^k_j\partial_j\sigma^k_i\|_{L^1_t(L^\infty)}$.
\end{lem}

\begin{proof}
 We will use the stohcastic flow and write
\begin{equation}
  \int_{\R^n} |f(t,x)|^p \dx = \int_{\R^n}|f_0(x)|^p|\det\partial\Psi_t(\Phi_t(x))|^{p-1}\dx.
\end{equation}
From the formula for $\det\partial\Psi_t(x)$ provided by Lemma \ref{lem:dpsi-rep}, it is a simple matter to conclude that for each $x\in \R^n$
\begin{equation}
    |\det\partial\Psi_t(\Phi_t(x))|^{p-1}\leq  C_{p,b,\sigma}\, \mathcal{E}_t(x)
  \end{equation}
where $\mathcal{E}_t(x)$ is the stochastic exponential of the martingale $(1-p)\int_0^t\Div\sigma^k(\Phi_s(x))\dee W_s^k$, specifically
\begin{equation}
  \mathcal{E}_t(x):= \exp\bigg\{(1-p) \int_0^t\Div\sigma^k(\Phi_{s}(x))\dee W^k_s- \frac{1}{2}(p-1)^2\int_0^t|\Div\sigma^k|^2(\Phi_{s}(x))\ds\bigg\}.
\end{equation}
Since $f_0(x)$ and $\mathcal{E}_t(x)$ are smooth in $x$, Fubini's theorem implies that since $\mathcal{E}_t(x)$ is a non-negative $\mathscr{F}_t$ martingale, then
\begin{equation}
  M_t := \int_{\R^n} |f_0(x)|^p\mathcal{E}_t(x)\dx
\end{equation}
is also a non-negative $\mathscr{F}_t$ martingale satisfying the SDE
\begin{equation}
  M_t = \int_{\R^n}|f_0(x)|^p\dx +  (1-p)\int_0^t\left(\int_{\R^n}|f_0(x)|^p\Div\sigma^k(s,\Phi_s(x))\mathcal{E}_s(x)\dx \,\right)\dee{W}_s^k
\end{equation}
and has quadratic variation
\begin{equation}
  \begin{aligned}
    \left[M\right]_t &= (p-1)^2\int_0^t\left(\int_{\R^n} |f_0(x)|^p\Div\sigma^k(s,\Phi_s(x))\mathcal{E}_s(x)\dx\right)^2\ds\\
    &\leq (p-1)^2\int_0^t\|\Div\sigma^k(s,\cdot)\|_{L^\infty}^2M_s^2\,\ds.
  \end{aligned}
\end{equation}
Using the BDG inequality, we deduce
\begin{equation}
  \begin{aligned}
    \E \left[\sup_{s\in[0,t]}\|f(s,\cdot)\|_{L^p}\right]^{2p}&\leq C_{p,b,\sigma} \E\left[ \sup_{s\in[0,t]}|M_t|^{2}\right]\\
      &\leq C_{p,b,\sigma} \int_0^t\|\Div\sigma^k(s,\cdot)\|_{L^\infty}^2\E \left[\sup_{r\in[0,s]} |M_r|^2\right]\,\ds
  \end{aligned}
\end{equation}
Using Gr\"{o}nwall, and the fact that $\|\Div\sigma^k(t,\cdot)\|_{L^\infty}^2$ is in $L^1_t$, gives the bound (\ref{eq:Lp-estimate-existence}).
\end{proof}

We are now equipped to prove Theorem \ref{thm:exist-unique}.

\begin{proof}[Proof of Theorem \ref{thm:existence}]
  The proof is straight forward. We first approximate $b,\sigma,f_0$ by smooth functions  $b_n$, $\sigma^k_n$, $f_{0,n}$ with compact support such that
  \begin{equation}
    \sup_n \left(\|\Div b_n\|_{L^1_t(L^\infty)} + \|\partial_i\sigma_{j,n}^k\partial_j\sigma^k_{i,n}\|_{L^1_{t}(L^\infty)} + \|\Div\sigma_n^k\|_{L^2_t(L^\infty)}\right) < \infty.
  \end{equation}
  and
  \begin{equation}
    b_n \to b, \quad \sigma^k_n \to \sigma^k,\quad f_{0,n} \to f_0,
  \end{equation}
  pointwise almost everywhere, with uniform bounds
  \begin{equation}
    \sup_{n}\|b_n\|_{L^1_t(L^{p^\prime})} < \infty,\quad  \sup_n\|\sigma^k_n\|_{L^2_t(L^{2p^\prime})} < \infty, \quad \sup_n\|f_{0,n}\|_{L^p} < \infty.
  \end{equation}
  This can, for instace be done by mollification and a smooth cutoff.

  Let $f_n$ be the unique classical solution to the stochastic continuity equation (\ref{eq:stochastic-transport-eq-1}) associated to $b_n$, $\sigma^k_n$, $f_{0,n}$. Using the estimate in Lemma \ref{lem:apriori-est} we conclude that $\{f_n\}$ is uniformly bounded in $L^2(\Omega\times[0,T]; L^p)$. Therefore $\{f_n\}$ has a  subsequence, still denoted $\{f_n\}$, and a limit $f$  such that
  \begin{equation}
    f_n \to f \quad \text{weakly in }\quad L^2(\Omega\times[0,T]; L^p)
  \end{equation}
Moreover since the space of progressively measurable processes in $L^2(\Omega\times[0,T];L^p)$ is closed (therefore weakly closed) and since $f_n(t,x)$ can easily seen to be a progessively measurable process with values in $L^P$ it follows that the limit $f$ is also progressively measurable.

  We now wish to pass the limit in the weak form. Let $\phi\in C^\infty([0,T]\times\R^n)$ such that $\phi$ is compactly supported in $x$, and $\phi(T,x) = 0$, then for each $n\geq 0$ we have
  \begin{equation}\label{eq:space-time-weak}
    \begin{aligned}
      &\int_0^T\left(\int_{\R^n}(\partial_t\phi + b_n\cdot\nabla\phi + \sigma^k_{i,n}\sigma^k_{j,n}\partial_i\partial_j\phi)f_n\dx\right)\dt\\
      &\hspace{1in}+ \int_0^T\left(\int_{\R^n}(\sigma^k_n\cdot\nabla\phi)f_n\dx\right)\dee W^k(t) = \int_{\R^n} \phi(0,\cdot)\,f_{0,n}\,\dx.
    \end{aligned}
  \end{equation}
  First, the pointwise convergence of $f_{n,0}$ and uniform bound in $L^p$ imply that
  \begin{equation}
    \int_{\R_n} \phi(0,\cdot)f_{0,n}\dx \to \int_{\R^n} \phi(0,\cdot) f_0\dx.
  \end{equation}
Also, by the pointwise convergence of $b_n$ and $\sigma_n$,
    \begin{equation}
      (\partial_t\phi + b_n\cdot\nabla\phi + \sigma^k_{i,n}\sigma^k_{j,n}\partial_i\partial_j\phi) \to (\partial_t\phi + b\cdot\nabla\phi + \sigma^k_{i}\sigma^k_{j}\partial_i\partial_j\phi)
    \end{equation}
    pointwise a.e on $[0,T]\times\R^n$ and uniform bounds on $b_n$ and $\sigma^k\tensor\sigma^k_n$ in $L^1_t(L^{p^\prime})$ imply, by the product limit Lemma \ref{lem:prod-lim}, that
    \begin{equation}
      \begin{aligned}
        &\int_0^T\left(\int_{\R^n}(\partial_t\phi + b_n\cdot\nabla\phi + \sigma^k_{i,n}\sigma^k_{j,n}\partial_i\partial_j\phi)f_n\dx\right)\dt\\
        &\hspace{1in}\to \int_0^T\left(\int_{\R^n}(\partial_t\phi + b\cdot\nabla\phi + \sigma^k_{i}\sigma^k_{j}\partial_i\partial_j\phi)f\dx\right)\dt
      \end{aligned}
    \end{equation}
    weakly in $L^2(\Omega)$.
    
It remains to pass the limit in the stochastic integral. Clearly,
  \begin{equation}
    \int_{\R^n} \sigma^k_n\cdot\nabla\phi\, f_n\dx \to \int_{\R^n}\sigma^k\cdot\nabla \phi\, f\dx\quad \text{ weakly in } L^2(\Omega\times[0,T]),
  \end{equation}
  and since the stochastic integral is a weakly continuous linear mapping from $L^2(\Omega\times[0,T])$ to $L^2(\Omega)$, we may pass the limit term by term in the summation for the stochastic integral. 

To obtain the almost sure, time integrated form, we remark that $\P$ almost surely,
\begin{equation}\label{eq:almost-sure-space-time-weak}
   \int_0^T\int_{\R^n}(\partial_t\phi + b\cdot\nabla + \sigma^k_i\sigma^k_j\partial_i\partial_j\phi)f\dx\dt + \sum_k\int_0^T\int_{\R^n}(\sigma^k\cdot\nabla\phi)f\dx\dee W^k_t = 0.
  \end{equation}
Now fix a $t\in [0,T]$ and choose a sequence of test functions $\phi^n(s,x) = \varphi(x)\psi^n(s)$, where $\psi^n(s)$ is a smooth approximation of the indicator $\1_{[0,t]}(s)$ so that $\partial_s \psi^n(s)$ is symmetric approximation of a delta function centered at $s=t$. Using the integrability of $f$ in time, Lebesgue's differentiation theorem implies that for almost every $t\in[0,T]$,
\begin{equation}
  \int_0^T \int_{\R^n}\partial_s\phi^n(s,x)f(s,x)\dx\ds \to \langle f(t), \varphi\rangle.
\end{equation}
Passing the limit in (\ref{eq:almost-sure-space-time-weak}), for test function $\phi = \phi^n$, gives the time-integrated weak form (\ref{eq:time-integrated-weak-form}).

Finally, using a standard application of the Kolmogorov continuity theorem, we can obtain that $\E\|\langle \varphi, f\rangle\|_{C_t} < \infty$ and therefore $t\mapsto \langle \varphi,f(t)\rangle$ has a continuous modification. It then follows from the equation for the weak, time-integrated form (\ref{eq:time-integrated-weak-form}), and the fact that $f$ is progressively measurable that the modification of $t\mapsto \langle \varphi,f(t)\rangle$ is also an $\mathscr{F}_t$ semi-martingale.
\end{proof}

\subsection{Uniqueness and Stability}

In this section, we will now consider the uniqueness of solutions
\begin{thm}[Uniqueness]\label{thm:uniqueness}
  Let $p\in [2,\infty]$ and suppose that
  \begin{equation}
    b\in L^1_t(W^{1,p^\prime}_\loc), \quad \sigma^k \in L^2_t(W^{1,\gamma}_\loc),\quad \gamma = \frac{2p}{p-2}
  \end{equation}
  and
  \begin{equation}
    \frac{b}{1+|x|} \in L^1_t(L^\infty),\quad \frac{\sigma^k}{1+|x|} \in L^2_t(L^\infty).
  \end{equation}
Then any weak $L^p$ solution to (\ref{eq:stochastic-transport-eq-1}) is unique among measurable functions on $\Omega\times[0,T]\times\R^n$.
\end{thm}

In order to prove the theorem, we will prove the following estimate
\begin{lem}\label{lem:stability}
  Let $p\in [2,\infty]$ and suppose that
  \begin{equation}
    b\in L^1_t(W^{1,p^\prime}_\loc), \quad \sigma^k \in L^2_t(W^{1,\gamma}_\loc),\quad \gamma = \frac{2p}{p-2}
  \end{equation}
  Then any weak $L^p$ solution $f$ satisfies for each $\varphi\in C^\infty_c(\R^n)$, and $t\in [0, T]$,
  \begin{equation}\label{eq:L^1-stab}
    \begin{aligned}
      &\E \int_{\R^n} \varphi(x) |f(t,x)|\dx = \int_{\R^n}\varphi(x)|f_0(x)|\dx\\
      &\hspace{2em}+ \E \int_0^t\left(\int_{\R^n} \left(b\cdot\nabla \varphi + \frac{1}{2}\sigma^k_i\sigma^k_j\partial_i\partial_j\varphi\right)(s,x)|f(s,x)|\dx\right)\ds
    \end{aligned}
  \end{equation}
\end{lem}
\begin{proof}
  Since $b$ and $\sigma^k$ are sufficiently regular, we begin by renormalizing the equation and applying Theorem \ref{thm:main-renorm-result-generalized} for a smooth renormalized $\Gamma(z)$. Taking expectation gives
  \begin{equation}\label{eq:renormalized-expec}
    \begin{aligned}
      &\E \int_{\R^n} \varphi(x) \Gamma(f(t,x))\dx = \int_{\R^n}\varphi(x)\Gamma(f_0(x))\dx\\
      &\hspace{2em}+ \E \int_0^t\left(\int_{\R^n} \left(b\cdot\nabla \varphi + \frac{1}{2}\sigma^k_i\sigma^k_j\partial_i\partial_j\varphi\right)(s,x)\Gamma(f(s,x))\dx\right)\ds\\
      &\hspace{2em}+ \E\int_0^t\left(\int_{\R^n}\left(\frac{1}{2}\partial_i\sigma^k_j\partial_j\sigma^k_i\varphi - (\Div b)\varphi - \sigma^k(\Div\sigma^k)\varphi\right)(s,x)G(f(s,x))\dx\right)\ds\\
      &\hspace{2em}+ \E\int_0^t\left(\int_{\R^n}\frac{1}{2}(\Div \sigma^k)^2(s,x)\varphi(x)H(f(s,x))\dx\right)\ds.
    \end{aligned}
  \end{equation}
  We now aim to approximate the absolute value $|z|$ with smooth bounded function $\Gamma_\ep(z)$. For instance, we can choose $\Gamma_\ep(z)= A_\ep(z)B_\ep(z)$, where
  \begin{equation}
    A_\ep(z) =
    \begin{cases}
      \frac{1}{2\ep^2}z^2 + \frac{1}{2}\ep, & \quad |z|<\ep\\
      |z|, &\quad |z|>\ep,
    \end{cases}   
  \end{equation}
  is a parabolic approximation of $|z|$ and $B_\ep(z)$ is a smooth non-negative, non-increasing function which is $1$ on the set $\{|z| < \ep^{-1}\}$ and zero on the set $\{|z| > 2\ep^{-1}\}$, satisfying $B_\ep^\prime(z) \leq \ep$ and $|B^{\prime\prime}_\ep(z)| \leq \ep^2$.

  We let $G_\ep(z) = z\Gamma^\prime(z) - \Gamma_\ep(z)$ and $H_\ep(z) = zG^\prime_\ep(z) -  G_\ep(z)$. From the formulae
  \begin{equation}
    G_\ep(z) = \left(\frac{1}{2\ep}z^2 - \frac{1}{2}\ep\right)B_\ep(z)\1_{|z|<\ep}  + A_\ep(z) zB_\ep(z),
  \end{equation}
  \begin{equation}
    H_\ep(z) = \frac{1}{2}\ep B_\ep(z)\1_{|z|<\ep} + \left(\frac{1}{\ep}z^2 - \ep\right)zB_\ep^\prime(z)\1_{|z|<\ep} + A_\ep(z)\left(zB^\prime(z) + z^2 B^{\prime\prime}(z)\right),
  \end{equation}
  and the definition of $\Gamma_\ep(z)$, we can see that for each $z\in \R$,
  \begin{equation}
    \Gamma_\ep(z) \to |z|,\quad G_\ep(z) \to 0, \quad H_\ep(z) \to 0,
  \end{equation}
  as $\ep \to 0$, and the following bounds hold
  \begin{equation}
    \sup_{\ep} \Gamma_\ep(z) \leqc (1+|z|),\quad \sup_{\ep} |G_\ep(z)| \leqc (1+|z|),\quad  \sup_{\ep} |H_\ep(z)| \leqc (1+|z|).
  \end{equation}
  Using these bounds, and the fact that $f\in L^{1}(\Omega; L^\infty_t(L^1_{\loc}))$, we can use the and the product limit Lemma \ref{lem:prod-lim} in $t,x$ and dominated convergence in $\omega$ to pass the limit as $\ep\to 0$ in each term of (\ref{eq:renormalized-expec}) with $\Gamma(z)$ replaced by $\Gamma_\ep(z)$ to conclude the desired bound (\ref{eq:L^1-stab}).
\end{proof}

We are now ready to prove Theorem \ref{thm:uniqueness}.

\begin{proof}[Proof of Theorem \ref{thm:uniqueness}]
  First, we define the weight
  \begin{equation}
    w(x) := (1 + |x|^2)^{-r/2}
  \end{equation}
  where $r> n$ so that $w(x)$ is integrable on $\R^n$. Note that
  \begin{equation}
    |\nabla w(x)| \leqc \frac{w(x)}{1 + |x|},
  \end{equation}
  and therefore $|\nabla w(x)|$ is also integrable. Using the fact that
  \begin{equation}
    (1+|x|)^{-1}b \in L^1_{t}(L^\infty), \quad (1+|x|)^{-1}\sigma^k \in L^2_t(L^\infty),
  \end{equation}
  along with the fact that $w(x)|f(t,x)|$ belongs to $L^1(\Omega; L^\infty_t(L^1))$, we can approximate $w(x)$ by a sequence of smooth compactly supported functions $\varphi_\ep(x)$ and pass the limit in (\ref{eq:L^1-stab}) to obtain the estimate
  \begin{equation}
     \begin{aligned}
      &\E \int_{\R^n} w(x) |f(t,x)|\dx \leqc \int_{\R^n}w(x)|f_0(x)|\dx\\
      &\hspace{2em}+ \int_0^t\left\|\frac{b(s,x)}{1+|x|}\right\|_{L^\infty}\E\left(\int_{\R^n}w(x)|f(s,x)|\dx\right)\ds\\
      &\hspace{2em}+ \int_0^t\left\|\frac{\sigma^k(s,x)}{1+|x|}\right\|_{L^\infty}^2\E\left(\int_{\R^n}w(x)|f(s,x)|\dx\right)\ds.
    \end{aligned}
  \end{equation}
  Applying Gr\"{o}nwall gives the stability estimate
  \begin{equation}
    \sup_{t\in[0,T]}\E \int_{\R^n} w(x) |f(t,x)|\dx\ds \leq C_{b,\sigma} \int_{\R^n} w(x)|f_0(x)|\dx,
  \end{equation}
 where the constant $C_{b,\sigma}$ depends on $b$ and $\sigma^k$ through the norms $\|(1+|x|)^{-1}b\|_{L^1_t(L^\infty)}$ and $\|(1+|x|)^{-1}\sigma^k\|_{L^2_t(L^\infty)}$. It follows that uniqueness holds $\P\tensor\dt\tensor\dx$ almost everywhere on $\Omega\times[0,T]\times\R^n$.
\end{proof}

\section{Renormalized Solutions for Integrable Drift}\label{sec:renorm-sol-int-drift}

In this section, we turn to a study of the subtle regularizing effects of stochastic transport with regards to the well-posedness of the stochastic transport problem in $L^p$, particularly under circumstances where the corresponding deterministic problem is ill-posed. Our primary interest will be in stochastic transport equations with non-degenerate noise,
  \begin{equation}\label{eq:non-degenerate-SPDE}
    \partial_t f + \Div(b f) - \frac{1}{2}\Delta f + \partial_i f \dot{W}^i= 0.
  \end{equation}
We will specifically we interested in the case that $b \in L^q_t(L^p)$ where $p,q$ satisfy the Krylov-R\"{o}ckner condition
\begin{equation}\label{eq:Krylov-Rockner}
  q\geq 2, p > 2, \quad \frac{2}{q} + \frac{n}{p} < 1.
\end{equation}

Our goal in this section will be to prove Theorem \ref{thm:rough-diff-renorm-theorm}. To begin, we will need some results on an associated parabolic problem.

\subsection{An Associated Parabolic Problem}
This section concerns the following damped backward parabolic problem for each $\lambda > 0$,
\begin{equation}\label{eq:parabolic-problem}
  \partial_tu_\lambda + b\cdot\nabla u_\lambda + \frac{1}{2}\Delta u_\lambda = \lambda u_\lambda - b ,\quad u_\lambda|_{t=T} = 0.
\end{equation}
and $u_\lambda\in \R^n$.

The following regularity and well-posedness result comes from \cite{Fedrizzi2013-cz} (Theorem 3.3 p 716) based on a Theorem from \cite{Krylov2004-uu}.

\begin{thm}\label{thm:exist-unique}
  Let $b \in L^q_t(L^p)$, with $p,q$ satisfying (\ref{eq:Krylov-Rockner}), then there exists a unique solution $u_\lambda:[0,T]\times\R^n \to \R^n$ to (\ref{eq:parabolic-problem}) satisfying
  \begin{equation}
    \|\partial_t u_{\lambda}\|_{L^q_t(L^p)} + \|u_\lambda\|_{L^q_t(W^{2,p})} \leqc \|b\|_{L^q_t(L^p)}
  \end{equation}
\end{thm}

The following lemma was proven in \cite{Fedrizzi2013-cz} (Lemma 3.4 p 717) using a decay properties of the heat semi-group and a generalized Gr\"{o}nwall inequality.
\begin{lem}\label{lem:Lipschitz-decay}
  Let $u_\lambda$ be the unique solution to (\ref{eq:parabolic-problem}), with $b\in L^q_t(L^p)$ and $p,q$ satisfying (\ref{eq:Krylov-Rockner}), then the following limit holds
  \begin{equation}\label{eq:Lipschitz-decay}
    \lim_{\lambda\to \infty} \|u_{\lambda}\|_{L^\infty_t(W^{1,\infty})}  = 0
  \end{equation}
\end{lem}

The Lipschitz bound provided by Lemma \ref{lem:Lipschitz-decay} will be useful in proving a number of additional properties of $u_\lambda$. Specifically, the following uniform estimates, giving precise decay of $u_\lambda$ and it's derivatives in various norms.

\begin{lem}\label{lem:precise-decay}
  Let $u_\lambda$ be the unique solution to (\ref{eq:parabolic-problem}) with $b\in L^q_t(L^p)$ with $p,q$ satisfying (\ref{eq:Krylov-Rockner}). Let $\alpha \in \{0,1\}$ and $r\in [1,p]$, or $\alpha = 2$ and $r\in[1,p)$, then the following estimate holds
  \begin{equation}\label{eq:lamulam-bound}
    \sup_{\lambda >0}\lambda^{\delta} \|\nabla^{\alpha}u_{\lambda}\|_{L^q_t(L^r)} < \infty, \quad \text{for}\quad \delta = 1-\frac{\alpha}{2} + \frac{n}{2}\left(\frac{1}{r}- \frac{1}{p}\right)
  \end{equation}
\end{lem}
\begin{proof}
  In what follows, to simplify notation, we will instead prove the Lemma for the forwards parabolic problem
  \begin{equation}\label{eq:forward-parabolic}
    \partial_tu_{\lambda} - b\cdot\nabla u_{\lambda} - \frac{1}{2}\Delta u_\lambda  = b - \lambda u_\lambda,\quad u_{\lambda}|_{t=0} = 0,
  \end{equation}
  which is related to the backwards problem (\ref{eq:parabolic-problem}) by time-reversal $u_\lambda(T-t,x)$ and hence Theorem \ref{thm:exist-unique} and Lemma \ref{lem:Lipschitz-decay} and also apply to (\ref{eq:forward-parabolic}).

  We begin by writing (\ref{eq:forward-parabolic}) as
  \begin{equation}
    u_\lambda(t) = \int_0^t e^{-\lambda(t-s)}P_{t-s}(b(s) + b(s)\cdot\nabla u_{\lambda}(s))\ds
  \end{equation}
  Where $P_{t}$ denotes the heat semi-group, defined by
  \begin{equation}
    P_t\phi(x) = \frac{1}{(4\pi t)^{n/2}}\int_{\R^n} e^{-|x-y|^2/4 t}\phi(y)\dy.
  \end{equation}
  Using the well-known estimate
  \begin{equation}\label{eq:heat-group-est}
    \|\nabla^{\alpha}P_tg\|_{L^r} \leqc t^{- \frac{\alpha}{2} + \frac{n}{2}\left(\frac{1}{r} - \frac{1}{p}\right)}\|g\|_{L^p} = t^{\delta-1}\|g\|_{L^p}
  \end{equation}
  we may conclude
  \begin{equation}\label{eq:Lr-bound-mild}
    \|\nabla^\alpha u_{\lambda}(t)\|_{L^r}\leqc (1 +\|\nabla u_{\lambda}\|_{L^\infty_t(L^\infty)})\int_{0}^t e^{-\lambda(t-s)}(t-s)^{\delta-1}\|b(s)\|_{L^p}\ds
  \end{equation}
  Taking the $L^q$ norm in time and on both sides of the inequality, using Young's inequality and the fact that
  \begin{equation}
    \|e^{-\lambda t}t^{\delta-1}\|_{L^1_t} \leqc \lambda^{-\delta}
  \end{equation}
  implies
  \begin{equation}
    \|\nabla^\alpha u_{\lambda}(t)\|_{L^q_t(L^r)}\leqc \lambda^{-\delta}(1 +\|\nabla u_{\lambda}\|_{L^\infty_t(L^\infty)})\|b\|_{L^q_t(L^p)}.
  \end{equation}
  Multiplying both sides by $\lambda^\delta$, taking the supremum in $\lambda >0$ and applying Lemma \ref{lem:Lipschitz-decay} proves (\ref{eq:lamulam-bound}).
\end{proof}
We also have the following convergence properties of $\lambda u_\lambda$ as $\lambda \to \infty$.

\begin{lem}\label{lem:divlamu-conv}
  Let $u_\lambda$ and $b$ be as in Lemma \ref{lem:precise-decay}, then
  \begin{equation}\label{eq:lamulam}
    \lim_{\lambda \to \infty}\|\lambda u_\lambda - b\|_{L^1_t(L^p)} = 0.
  \end{equation}
  If in addition $b$ satisfies
  \begin{equation}
    \Div b \in L^1_t(L^1),
  \end{equation}
  then
  \begin{equation}\label{eq:div-lam-u}
    \lim_{\lambda\to \infty} \|\Div(\lambda u_{\lambda}) - \Div b\|_{L^{1}_t(L^1)} = 0.
  \end{equation}
\end{lem}

\begin{proof}
 As we did in Lemma \ref{lem:precise-decay}, we will instead study the forward problem (\ref{eq:forward-parabolic}) and it's mild form
  \begin{equation}
    u_\lambda(t) = \int_0^t e^{-\lambda(t-s)}P_{t-s} (b(s) + b(s)\cdot\nabla u_{\lambda}(s))\ds.
  \end{equation}
  We prove (\ref{eq:lamulam-bound}) first. We write
  \begin{equation}
      \lambda u_\lambda(t) - b(t) = \int_0^t\lambda e^{-\lambda(t-s)} P_{t-s}b(s) \ds - b(t) +\int_0^t\lambda e^{-\lambda(t-s)} P_{t-s}(b(s)\cdot\nabla u_{\lambda}(s))\ds.
    \end{equation}
    Using the fact that
    \begin{equation}
      \int_0^t \lambda e^{-\lambda(t-s)}\ds = 1-e^{-\lambda t},
    \end{equation}
    this can be written as
    \begin{equation}
      \begin{aligned}
        \lambda u_\lambda(t) - b(t) &= -e^{-\lambda t}b(t) + \int_0^t\lambda e^{-\lambda(t-s)} P_{t-s}(b(s) - b(t))\ds\\
        &\hspace{1in}+\int_0^t\lambda e^{-\lambda(t-s)} P_{t-s}(b(s)\cdot\nabla u_{\lambda}(s))\ds.
      \end{aligned}
    \end{equation}
    Taking the $L^p$ norm of both sides, we find
    \begin{equation}\label{eq:lamulam-b-bound}
      \begin{aligned}
        \|\lambda u_\lambda (t) - b(t) \|_{L^p} &\leqc e^{-\lambda t}\|b(t)\|_{L^p} + \int_0^t \lambda e^{-\lambda(t-s)}\|P_{t-s}b(s) - b(t)\|_{L^p}\ds\\
        &\hspace{1in}+ \int_0^t \lambda e^{-\lambda(t-s)} \|b(s)\|_{L^p}\|\nabla u_{\lambda}(s)\|_{L^\infty}\ds
    \end{aligned}
  \end{equation}
  Applying dominated convergence we find
  \begin{equation}
    e^{-\lambda t}\|b(t)\|_{L^p} \to 0 \quad \text{in}\quad L^1_t,
  \end{equation}
  in addition, the function $t\mapsto \lambda e^{-\lambda t}$ is an integrable approximation of the identity, and therefore
  \begin{equation}
    \int_0^t \lambda e^{-\lambda(t-s)}\|P_{t-s} b(s) - b(t)\|_{L^p}\ds \to \|P_{0}b(t) -b(t)\|_{L^p} = 0,\quad \text{in} \quad L^1_t.
  \end{equation}
  Therefore, taking the $L^1_t$ norm of both sides of (\ref{eq:lamulam-b-bound}) and  applying Young's inequality gives
  \begin{equation}
    \|\lambda u_{\lambda} - b(t)\|_{L^1_t(L^p)} \leqc o(1) + \|\lambda e^{-\lambda t}\|_{L^1_t}\|b\|_{L^q_t(L^p)}\|\nabla u_{\lambda}\|_{L^\infty_t(L^\infty)},
  \end{equation}
  Applying Lemma \ref{lem:Lipschitz-decay} to the above inequality and taking $\lambda \to \infty$ concludes the proof of (\ref{eq:lamulam}).

We now proceed to prove (\ref{eq:div-lam-u}). Similar to the proof of (\ref{eq:lamulam}) we can write
  \begin{equation}
    \begin{aligned}
      \Div (\lambda u_\lambda(t)) - \Div b(t)&= - e^{-\lambda t}\Div b(t) + \int_0^t \lambda e^{-\lambda(t-s)}\left[P_{t-s}\Div b(s) - \Div b(t)\right]\ds\\
      &\hspace{1in}- \int_0^t \lambda e^{-\lambda(t-s)}\Div P_{t-s}(\Div b(s) u_{\lambda}(s))\ds\\
      &\hspace{1in}+ \int_0^t \lambda e^{-\lambda(t-s)}\Div \Div P_{t-s}(b(s)\!\tensor\! u_{\lambda}(s))\ds\\
      &= I_{\lambda,1}(t) + I_{\lambda,2}(t) + I_{\lambda,3}(t) + I_{\lambda,4}(t).
      \end{aligned}
    \end{equation}
    We will now show that each of the terms $I_{\lambda,i}(t)$, $i \in \{1,2,3,4\}$ vanish in $L^1_t(L^r)$ as $\lambda \to \infty$. Upon showing this, the proof will be complete. Using the heat semi-group estimate (\ref{eq:heat-group-est}) and H\"{o}lder's inequality, we obtain the bounds
  \begin{equation}
    \begin{aligned}
      &\|I_{\lambda,1}(t)\|_{L^r} \leq e^{-\lambda t}\|\Div b(t)\|_{L^r},\\
      &\|I_{\lambda,2}(t)\|_{L^r} \leq \int_0^t \lambda e^{-\lambda(t-s)}\|P_{t-s}\Div b(s) - \Div b(t)\|_{L^r}\ds,\\
      &\|I_{\lambda,3}(t)\|_{L^r} \leqc \int_0^te^{-\lambda (t-s)} (t-s)^{- \frac{1}{2} - \frac{n}{2p}}\|\Div b(s)\|_{L^r}\|\lambda u_{\lambda}(s)\|_{L^p}\ds,\\
      &\|I_{\lambda,4}(t)\|_{L^r} \leqc \int_0^te^{-\lambda (t-s)}(t-s)^{-1+\frac{n}{2}\left(\frac{1}{r} - \frac{2}{p}\right)}\|b(s)\|_{L^p}\|\lambda u_{\lambda}(s)\|_{L^p}\ds.
    \end{aligned}
  \end{equation}
  Using the same dominated convergence and approximation of the identity argument for the proof of (\ref{eq:lamulam}) we obtain
  \begin{equation}
    \|I_{\lambda,1}\|_{L^1_t(L^r)}\to 0 \text{ and } \|I_{\lambda,2}\|_{L^1_t(L^r)} \to 0\text{ as }\lambda\to \infty.
  \end{equation}
 For $I_{\lambda,3}(t)$ and $I_{\lambda,4}(t)$, we use Young's inequality to obtain
  \begin{equation}\label{eq:I3-est}
    \begin{aligned}
      \|I_{\lambda,3}\|_{L^1_t(L^r)}&\leqc \|e^{-\lambda t} t^{-\frac{1}{2} -\frac{n}{2p}}\|_{L^{q^\prime}}\|\Div b\|_{L^1_t(L^r)}\|\lambda u_{\lambda}\|_{L^q_t(L^p)}\\
      &\leqc \lambda^{\delta_0} \|\Div b\|_{L^1_t(L^r)}\sup_{\lambda >0}\|\lambda u_{\lambda}\|_{L^q_t(L^p)},
    \end{aligned}
  \end{equation}
  where
  \begin{equation}
    \delta_0 = -\frac{1}{2} + \frac{1}{2}\left(\frac{2}{q} + \frac{n}{p}\right) < 0
  \end{equation}
  by condition (\ref{eq:Krylov-Rockner}), and similarly
  \begin{equation}\label{eq:I4-est}
    \begin{aligned}
      \|I_{\lambda,4}\|_{L^1_t(L^r)} &\leqc \|e^{-\lambda t}t^{-1 + \frac{n}{2}\left(\frac{1}{r} - \frac{2}{p}\right)}\|_{L^{q^\prime/2}_t}\|b\|_{L^q_t(L^p)}\|\lambda u_{\lambda}\|_{L^q_t(L^p)}\\
      &\leqc\lambda^{\delta_1}\|b\|_{L^q_t(L^p)}\sup_{\lambda >0}\|\lambda u_{\lambda}\|_{L^q_t(L^p)},
     \end{aligned}
   \end{equation}
   where
   \begin{equation}
     \delta_1 = -1 + \frac{q}{2} - \frac{n}{2}\left(\frac{1}{r} - \frac{2}{p}\right) \leq -1 + \frac{q}{2} < 0,
   \end{equation}
   also by condition (\ref{eq:Krylov-Rockner}).
   
  It follows from the uniform bound (\ref{eq:lamulam-bound}) of Lemma \ref{lem:divlamu-conv} with $\alpha = 0$ and $r =p$ , that the right-hand side of both (\ref{eq:I3-est}) and (\ref{eq:I4-est}) vanish as $\lambda \to \infty$.
\end{proof}

\subsection{The Transformed Problem}

The backwards parabolic problem (\ref{eq:parabolic-problem}) can be used to transform the SPDE. Indeed, the solution $u_{\lambda}$ can be used to define a family of diffeomorphisms $\phi^{\lambda}_t:\R^n\to \R^n$, defined by
\begin{equation}\label{eq:phi-trans}
  \phi^{\lambda}_t(x) = x + u_\lambda(t,x).
\end{equation}
The diffeomorphism property of $\phi^{\lambda}_t$ is only valid for $\lambda > 0$ large enough and follows from Lemma (\ref{eq:Lipschitz-decay}) since for each $\ep \in (0,1)$, there exists a $\lambda > 0$ such that
\begin{equation}
  \|\nabla u_{\lambda}\|_{L^\infty_t(L^\infty)} < \ep.
\end{equation}
As a consequence, $\nabla\phi_t^\lambda(x) = I + \nabla u_{\lambda}(t,x)$ is an invertible matrix, and satisfies
\begin{equation}\label{eq:det-bound}
  (1-\ep)^n \leq  \det\left(I + \nabla u_\lambda\right) \leq (1+\ep)^n.
\end{equation}

Our strategy will be to transform the solution $f$ to the SPDE (\ref{eq:non-degenerate-SPDE}) by pushing $f(t)$ forward under the diffeomorphism $\phi_t^\lambda$,
\begin{equation}
  h_\lambda(t) = (\phi^{\lambda}_t){}_{\#}f(t).
\end{equation}
Using the equation for $u_\lambda(t)$ one can show that $h_\lambda(t)$ now solves a new It\^{o} type stochastic continuity equation, with a more regular drift. This is achieved at the expense of introducing new diffusion coefficients with minimal regularity, albeit enough regularity to apply the results of Section \ref{sec:renorm-rough-diff}.

Our first step will be determine the equation satisfied by $h_\lambda(t)$ 
\begin{prop}\label{prop:transform-prob}
  Let $f$ be a weak $L^{p^\prime}$ solution to (\ref{eq:non-degenerate-SPDE}) with $b\in L^q_t(L^p)$ with $p$ and $q$ satisfying (\ref{eq:Krylov-Rockner}). In addition let $u_{\lambda}$ be the solution to (\ref{eq:parabolic-problem}) with $\lambda>0$ large enough to that the mapping $x\mapsto \phi_t^\lambda(x) = x + u_\lambda(t,x)$ is invertible. Then the push-forward $h_\lambda(t) = (\phi^\lambda_t)_{\#}f(t)$ is a weak $L^{p^\prime}$ solution to \begin{equation}\label{eq:transformed-stoch-cont}
  \partial_t h_\lambda + \Div(\hat{b}_\lambda h_\lambda) - \frac{1}{2}\partial_i\partial_j(\hat{\sigma}_{\lambda,i}^k\hat{\sigma}_{\lambda,j}^kh_\lambda) + \Div(\hat{\sigma}_\lambda^k h_\lambda)\dot{W}^k = 0, 
\end{equation}
where $\hat{b}_\lambda$ and $\hat{\sigma}^k_\lambda$ are given by
\begin{equation}\label{eq:uhat-sighat}
  \begin{aligned}
    &\hat{b}_\lambda(t,x) = \lambda u_{\lambda}\!\left(t,(\phi^{\lambda}_t){}^{-1}(x)\right)\\ &\hat{\sigma}^k_\lambda(t,x) = \partial_k\phi^{\lambda}_{t}((\phi_{t}^{\lambda})^{-1}(x)).
    \end{aligned}
\end{equation}
\end{prop}

\begin{proof}
  The strategy will be to choose a time-dependent test function of the form $\hat{\varphi}(t,x) = \varphi(\phi_t^\lambda(x))$ for equation (\ref{eq:non-degenerate-SPDE}), where $\varphi$ is a smooth compactly supported test function. Of course, since $\phi_t^\lambda(x)$ is not a $C^\infty$ function, $\hat{\varphi}(t,x)$ is not a valid test function, and extra care must be taken to justify it's use in the weak form. This can be remedied by instead studying the mollified equation
  \begin{equation}
    \partial_t f_\ep(t,x) + \Div(b f)_\ep(t,x) - \frac{1}{2}\Delta f_\ep(t,x) + \partial_i f_\ep(t,x) \dot{W}^i = 0.
  \end{equation}
  It is important to note that $f_{\ep}(t,x)$ is defined everywhere and continuous in $x$, and by a straight-forward continuity estimate, has a version which is continuous in time.

  As a consequence we may apply the product rule (It\^{o}'s formula) to the process $t\mapsto f_\ep(t,x)\hat{\varphi}(t,x)$, yielding
  \begin{equation}
    f_\ep(t,x)\hat{\varphi}(t,x) - f_\ep(0,x)\hat{\varphi}(t,x) = \int_0^t f_\ep(s,x)\partial_s\hat{\varphi}(s,x)\ds+ \int_0^t \hat{\varphi}(s,x)\dee f_\ep(s,x).
  \end{equation}
  Using the equation for $f_\ep(t,x)$, integrating in space and time, applying Fubini's theorem to interchange space and time (as well as stochastic) integrals, and integrating by parts gives
  \begin{equation}\label{eq:time-dep-phi-ep}
    \begin{aligned}
      &\langle f_\ep(t),\hat{\varphi}(t)\rangle - \langle f_\ep(0),\varphi\rangle = \int_0^t \langle f_\ep(s),\partial_s \hat{\varphi}(s)\rangle\ds + \int_0^t\langle (f(s) b(s))_{\ep}\cdot\nabla  \hat{\varphi}(s)\rangle \ds \\
      &\hspace{1in}+\frac{1}{2}\int_0^t \langle f_\ep(s), \Delta \hat{\varphi}(s)\rangle\ds + \int_{0}^t \langle f_\ep(s), \partial_i\hat{\varphi}(s)\rangle \dee W^i(s).
    \end{aligned}
  \end{equation}
  It is important to remark that the Lipschitz property of $\phi^\lambda_t(x)$ implies that $\hat{\varphi}(t,x)$ has compact support in $x$ if $\varphi$ does. Moreover, the definition of $\phi_t^\lambda(x)$, Theorem \ref{thm:exist-unique} and Lemma \ref{lem:Lipschitz-decay} imply that
  \begin{equation}
    \begin{aligned}
      &\hat{\varphi}(t) \in L^\infty, \quad \partial_t\hat{\varphi} \in L^q_t(L^p)\\
      &\nabla \hat{\varphi} \in L^\infty_t(L^\infty), \quad \nabla^2\hat{\varphi} \in L^q_t(L^p).
    \end{aligned}
  \end{equation}
  As a consequence, all integrals in (\ref{eq:time-dep-phi-ep}) are well-defined. Furthermore, the fact that
  \begin{equation}
    \E\|f\|_{L^\infty_t(L^{p^\prime})}^r<\infty, \quad \text{for all}\quad r\geq 1.
  \end{equation}
  Allows the passage of the limit as $\ep \to 0$ in each term of (\ref{eq:time-dep-phi-ep}). One obtains
  \begin{equation}\label{eq:time-dep-phi}
    \begin{aligned}
      &\langle f(t),\hat{\varphi}(t)\rangle - \langle f(0),\varphi\rangle = \int_0^t \left\langle f(s),\partial_s \hat{\varphi}(s)+ b(s)\cdot\nabla \hat{\varphi}(s) + \tfrac{1}{2}\Delta\hat{\varphi}(s)\right\rangle\ds\\
        &\hspace{1in}+ \int_{0}^t \langle f(s), \partial_i\hat{\varphi}(s)\rangle \dee W^i(s).
    \end{aligned}
  \end{equation}
  We now use the equation for $u_\lambda(t,x)$ (\ref{eq:parabolic-problem}), to deduce that
    \begin{equation}
  \partial_t \phi^{\lambda}_t + b(t)\cdot\nabla \phi^{\lambda}_t + \frac{1}{2}\Delta\phi^{\lambda}_t = \lambda u_{\lambda}(t).
\end{equation}
and therefore
  \begin{equation}
    \begin{aligned}
      &\partial_t \hat{\varphi}(t) + b(t)\cdot\nabla\hat{\varphi}(t) + \frac{1}{2}\Delta \hat{\varphi}(t)\\
      &\hspace{.5in}= \lambda u_{\lambda}(t)\cdot \nabla \varphi(\phi^\lambda_t) + \frac{1}{2}\left(\partial_k\phi^\lambda_t\!\tensor\!\partial_k\phi^\lambda_t\right) : \nabla^2\varphi(\phi^\lambda_t).
    \end{aligned}
  \end{equation}
Substituting this into equation~(\ref{eq:time-dep-phi}) gives
  \begin{equation}\label{eq:transform-eq-non-push}
    \begin{aligned}
      &\langle f(t),\varphi(\phi^\lambda_t)\rangle - \langle f(0),\varphi\rangle\\
      &\hspace{.5in}= \int_0^t \left\langle f(s), \lambda u_{\lambda }(s) \cdot \nabla \varphi(\phi^\lambda_s) + \tfrac{1}{2}\left(\partial_k\phi^\lambda_s\!\tensor\!\partial_k\phi^\lambda_s\right) : \nabla^2\varphi(\phi^\lambda_s)\right\rangle\ds\\
        &\hspace{.5in}+ \int_{0}^t \langle f(s), \partial_i\phi^\lambda_s\cdot \nabla\varphi(\phi^\lambda_s)\rangle \dee W^i(s).
    \end{aligned}
  \end{equation}
  Finally, using the fact that the push-forward $h_\lambda(t) = (\phi_t^\lambda)_{\#}f(t)$ satisfies
  \begin{equation}
    \langle f(t)g, \varphi(\phi_t^\lambda)\rangle= \langle h_\lambda(t)g((\phi^\lambda_t)^{-1}),\varphi\rangle,
  \end{equation}
  for any $g\in L^p$, gives
  \begin{equation}\label{eq:transform-eq-final-form}
    \begin{aligned}
      &\langle h_\lambda(t),\varphi\rangle - \langle f(0),\varphi\rangle\\
      &\hspace{.5in}= \int_0^t \left\langle h_\lambda(s), \hat{b}_\lambda(s) \cdot \nabla \varphi + \tfrac{1}{2}\left(\hat{\sigma}_\lambda^k(s)\!\tensor\!\hat{\sigma}_\lambda^k(s)\right) : \nabla^2\varphi\right\rangle\ds\\
        &\hspace{.5in}+ \int_{0}^t \langle h_\lambda(s), \hat{\sigma}_\lambda^k(s)\cdot\nabla\varphi\rangle \dee W^i(s),
    \end{aligned}
  \end{equation}
  where $\hat{b}_\lambda$ and $\hat{\sigma}^k_\lambda$ are given by (\ref{eq:uhat-sighat}). This is precisely the weak form of equation (\ref{eq:transformed-stoch-cont}).
\end{proof}

\subsection{Relaxation of the Transformed Problem}

In order to deduce renormalizability for $f$ the strategy will be to send $\lambda \to \infty$, relaxing the transformed process $h_\lambda$ back to $f$ in the limit. In order to do this, we will need the following convergence results.
\begin{lem}\label{lem:Relaxation-Lem}
  Let $f$, $\phi_\lambda$ and $h_\lambda$ be as in Proposition \ref{prop:transform-prob} and let $\hat{b}_\lambda$ and $\hat{\sigma}^k_\lambda$ be given by (\ref{eq:uhat-sighat}).
Then, the following limits hold

\begin{equation}\label{eq:h-lam-conv}
h_\lambda \to f
\end{equation}
$\P\tensor\dt\tensor\dx$ almost everywhere on $\Omega\times[0,T]\times\R^n,$
\begin{align}
    \label{eq:u-lam-conv}\hat{b}_\lambda \to b \quad &\text{in}\quad L^q_t(L^p),\\
  \label{eq:sig-lam-conv}\hat{\sigma}^k_{\lambda,i} \to \delta_{i,k}\quad &\text{in}\quad L^q_t(L^{p}),\\
  \label{eq:sig-lam-nab-conv}\nabla\hat{\sigma}^k_{\lambda} \to 0\quad &\text{in}\quad L^q_t(L^r), \quad\text{for each}\quad r< p.
\end{align}
If, in addition,
\begin{equation}
  \Div b \in L^1_t(L^1),
\end{equation}
then
\begin{equation}\label{eq:div-u-lam-conv}
  \Div\hat{b}_\lambda \to \Div b\quad \text{in}\quad L^1_t(L^1).
\end{equation}
\end{lem}
\begin{proof}
For the first limit (\ref{eq:h-lam-conv}), recall the push-forward is given by
  \begin{equation}
    h_\lambda(t,x) = f(t,(\phi^{\lambda}_t)^{-1}(x))\det\left(\nabla(\phi^{\lambda}_t)^{-1}(x)\right).
  \end{equation}
  It follows from (\ref{eq:Lipschitz-decay}) that, pointwise almost surely in $\R^n$,
  \begin{equation}
    (\phi^\lambda_t)^{-1}(x)\to x, \quad \det\left(\nabla(\phi^{\lambda}_t)^{-1}(x)\right) \to 1,
    \end{equation}
 therefore limit (\ref{eq:h-lam-conv}) follows. The limit (\ref{eq:u-lam-conv}) follows from the identity
  \begin{equation}
      \hat{b}_\lambda(t,x) = \lambda u_{\lambda}\left((\phi^\lambda_t)^{-1}(x)\right)
    \end{equation}
    after applying Lemmas \ref{lem:divlamu-conv} and \ref{lem:diffeo-lam-conv}. The limit (\ref{eq:sig-lam-conv}) follows similarly by writing
    \begin{equation}
      \hat{\sigma}^k_{\lambda,i}(t,x) = \delta_{k,i} + \partial_ku_{\lambda}\left((\phi^\lambda_t)^{-1}(x)\right),
    \end{equation}
    and applying the following result of Lemma \ref{lem:precise-decay} with $\alpha = 1$ and $r= p$,
    \begin{equation}
      \|\partial_ku_{\lambda}\left((\phi^\lambda_t)^{-1}(x)\right)\|_{L^q_t(L^p)} \leqc \|\partial_k u_{\lambda}\|_{L^q_{t}(L^p)} \leqc \lambda^{-\frac{1}{2}} \to 0.
    \end{equation}
In order to prove the limit (\ref{eq:sig-lam-nab-conv}), we first write 
\begin{equation}\label{eq:nabla-sigma-k-id}
    \nabla\hat{\sigma}^k_\lambda(t,x) = \left[ I + \nabla u_{\lambda}(t,(\phi_t^\lambda)^{-1}(x))\right]^{-1}\left[\partial_k\nabla u_{\lambda}\left(t,(\phi_t^\lambda)^{-1}(x)\right)\right].
  \end{equation}
  We may assume that $\lambda$ is large enough so that
  \begin{equation}
    \|\nabla u_{\lambda}\|_{L^\infty_t(L^\infty)} < \frac{1}{2}.
  \end{equation}
  It follows that
  \begin{equation}
    [I - \nabla u_{\lambda}(t,(\phi^{\lambda}_t)^{-1}(x))]^{-1} = \sum_{j\geq 0} [-\nabla u_{\lambda}(t,\phi^{\lambda}_t(x))]^j,
  \end{equation}
  and therefore
  \begin{equation}\label{eq:I-nabu-bound}
      \|[I - \nabla u_{\lambda}(t,(\phi^{\lambda}_t)^{-1}(x))]^{-1}\|_{L^\infty_t(L^\infty)} \leq \frac{1}{1- \|\nabla u_{\lambda}(t,(\phi^{\lambda}_t)^{-1}(x))\|_{L^\infty_t(L^\infty)}} \leq 2.
    \end{equation}
  The bound (\ref{eq:I-nabu-bound}) as well as the following consequence of Lemma \ref{lem:precise-decay} with $\alpha = 2$ and $r < p$,
  \begin{equation}
    \|\partial_k\nabla u_{\lambda}\left((\phi^\lambda_t)^{-1}(x)\right)\|_{L^q_t(L^r)}\leqc  \|\partial_k\nabla u_{\lambda}\|_{L^q_t(L^r)} \leqc \lambda^{-\frac{n}{2}\left(\frac{1}{r} - \frac{1}{p}\right)} \to 0,
  \end{equation}
 implies  limit (\ref{eq:sig-lam-nab-conv}).

  The last limit (\ref{eq:div-u-lam-conv}) is the most subtle. We begin by writing
  \begin{equation}
    \begin{aligned}
      \Div\hat{b}_{\lambda}(t,x) &= \lambda \Tr\left(\left[\nabla u_{\lambda}(t,(\phi^{\lambda}_t)^{-1}(x))\right]\left[I+\nabla u_{\lambda}(t,(\phi^{\lambda}_t)^{-1}(x))\right]^{-1}\right)\\
      &= \Div(\lambda u_{\lambda})(t,(\phi^{\lambda}_t)^{-1}(x))\\
      &\hspace{.5in}+ \lambda \Tr\left(\left[\nabla u_{\lambda}(t,(\phi^{\lambda}_t)^{-1}(x))\right]\left(\left[I+\nabla u_{\lambda}(t,(\phi^{\lambda}_t)^{-1}(x))\right]^{-1}- I\right)\right)\\
      &= \Div(\lambda u_{\lambda})(t,(\phi^{\lambda}_t)^{-1}(x))\\
      &\hspace{.5in}- \lambda \Tr\left(\left[\nabla u_{\lambda}(t,(\phi^{\lambda}_t)^{-1}(x))\right]^2\left[I+\nabla u_{\lambda}(t,(\phi^{\lambda}_t)^{-1}(x))\right]^{-1}\right).
      \end{aligned}
    \end{equation}
    The first term above converges to $\Div b$ in $L^1_t(L^1)$ by Lemma \ref{lem:divlamu-conv} and Lemma \ref{lem:diffeo-lam-conv}. We can deal with the second term by applying the bound (\ref{eq:I-nabu-bound}), and using Lemma \ref{lem:precise-decay} to conclude
    \begin{equation}
      \lambda\|[\nabla u_{\lambda}(t,(\phi^{\lambda}_t)^{-1}(x))]^2\|_{L^1_t(L^1)} \leqc \lambda\|[\nabla u_{\lambda}]^2\|_{L^1_t(L^1)} \leq \lambda \|\nabla u_{\lambda}\|^2_{L^q_t(L^{2})} \leqc \lambda^{\delta_0}
    \end{equation}
    where
    \begin{equation}
      \delta_0 = - n\left(\frac{1}{2} - \frac{1}{p}\right) < 0.
    \end{equation}
    Again applying Lemma \ref{lem:diffeo-lam-conv} concludes the proof of the limit (\ref{eq:div-u-lam-conv}).
  \end{proof}

  \subsection{Renormalization and Proof of Theorem \ref{thm:main-theorem}}
 We are now in a position to apply the results of Section \ref{sec:renorm-rough-diff} to the transformed problem (\ref{eq:transformed-stoch-cont}). Indeed, if $f$ is a weak $L^{p^\prime}$ solution to (\ref{eq:non-degenerate-SPDE}), then by Proposition \ref{prop:transform-prob} the push-forward $h_\lambda(t) = (\phi^\lambda_t)_{\#}f(t)$ is a weak $L^{p^\prime}$ solution to (\ref{eq:transformed-stoch-cont}). It is also a straightforward application of the formulas (\ref{eq:uhat-sighat}) and the regularity results of Theorem \ref{thm:exist-unique} and Lemma (\ref{eq:Lipschitz-decay}) to conclude that
\begin{equation}
  \hat{b}_\lambda \in L^q_t(W^{2,p}), \quad \hat{\sigma}^k_\lambda \in L^q_t(W^{1,p}).
\end{equation}
It is worth remarking that since 
\begin{equation}
  1- \frac{d}{p} \equiv \alpha > 0,
\end{equation}
then, by a Sobolev embedding we have
\begin{equation}
  \hat{b}_{\lambda} \in C^{1,\alpha},\quad \hat{\sigma}^k_\lambda\in C^{0,\alpha}.
\end{equation}
Therefore, while the drift $\hat{b}_\lambda$ is nice and regular, the diffusion coefficients $\hat{\sigma}_\lambda^k$ are not regular enough to apply classical strong existence and uniqueness techniques to the corresponding SDE.

However, we are in a regime where we can apply the results of Theorem \ref{thm:main-renorm-result} to obtain renormalizability of $h_\lambda$. Then, using the relaxation properties of Lemma \ref{lem:Relaxation-Lem}, we will then pass $\lambda \to \infty$ in the renormalized form for $h_{\lambda}$.

\begin{proof}[Proof of Theorem \ref{thm:main-theorem}]
    Recall from Proposition \ref{prop:transform-prob}, the  transformed solution $h_\lambda(t) = (\phi_t^\lambda)_{\#}f(t)$ is a weak $L^{p^\prime}$ solution to the equation
    \begin{equation}
      \partial_t h_\lambda + \Div(\hat{b}_\lambda h_{\lambda}) - \frac{1}{2}\partial_i\partial_j(\hat{\sigma}^k_{\lambda,i}\hat{\sigma}^k_{\lambda,j} h_{\lambda}) + \Div(\hat{\sigma}^k_{\lambda} h_{\lambda})\dot{W}^k = 0.
    \end{equation}
    More, as discussed at the beginning of this section, $\hat{b}_{\lambda}$ and $\hat{\sigma}_{\lambda}^k$ satisfy
    \begin{equation}
  \hat{b}_\lambda \in L^q_t(W^{2,p}), \quad \hat{\sigma}^k_\lambda \in L^q_t(W^{1,p}).
\end{equation}
We may conclude that $h_\lambda$ satisfies the renormalized equation
\begin{equation}\label{eq:renorm-hlam}
  \begin{aligned}
    &\langle \Gamma(h_\lambda(t)),\varphi\rangle = \langle \Gamma(h_\lambda(0)),\varphi\rangle + \int_0^t\langle  \Gamma(h_\lambda)\hat{b}_\lambda,\nabla\varphi\rangle\ds\\
    &\hspace{.5in}+ \frac{1}{2}\int _0^t\langle\Gamma(h_\lambda)\hat{\sigma}^k_{\lambda,i}\hat{\sigma}_{\lambda,j}^k,\partial_i\partial_j\varphi\rangle\ds + \int_0^t\langle\Gamma(h_\lambda)\hat{\sigma}_\lambda^k,\nabla\varphi\rangle \dee W^k\\
    &\hspace{.5in}- \int_0^t\langle G(h_\lambda)\Div \hat{\sigma}_\lambda^k,\varphi\rangle\dee{W}^k -\int_0^t\langle G(h_\lambda)\Div \hat{b}_\lambda,\varphi\rangle\ds\\
    &\hspace{.5in}- \int_0^t\langle G(h_\lambda)(\Div\hat{\sigma}_\lambda^k)\hat{\sigma}_\lambda^k,\nabla\varphi\rangle\ds + \frac{1}{2}\int_0^t\langle G(h_\lambda)\partial_i\hat{\sigma}^k_{\lambda,j}\partial_j\hat{\sigma}^k_{\lambda,i},\varphi\rangle\ds\\
    &\hspace{.5in}+ \frac{1}{2}\int_0^t\langle H(h_\lambda)(\Div\sigma_\lambda^k)^2,\varphi\rangle\ds.
  \end{aligned}
\end{equation}
We will assume at this point, that $\Gamma(z)$ is $C^2$ and satisfies
\begin{equation}\label{eq:Gamma-bound}
  \sup_{z} \left(|\Gamma(z)| +  |z\Gamma^{\prime}(z)| +|z^2\Gamma^{\prime\prime}(z)|\right) <\infty.
\end{equation}
As a consequence $G(z)$ and $H(z)$ are bounded functions. Our goal will now be to pass the limit as $\lambda \to \infty$ in the weak form of (\ref{eq:renorm-hlam}) as $\lambda \to \infty$ using Lemma \ref{lem:Relaxation-Lem}. 

To begin, we remark that the pointwise convergence of $h_\lambda$ given in (\ref{eq:h-lam-conv}) along with the fact that
\begin{equation}
  \sup_{\lambda} \|\Gamma(h_{\lambda})\|_{L^\infty_{t,x}} < \infty
\end{equation}
allows us to use the product limit Lemma \ref{lem:prod-lim} in conjunction with Lemma \ref{lem:Relaxation-Lem} to obtain the following $\P$ almost sure limits
\begin{equation}
  \begin{aligned}
    \langle \Gamma(h_\lambda)\hat{b}_\lambda ,\nabla\varphi\rangle &\to \langle \Gamma(f)b , \nabla \varphi\rangle,\quad \text{in}\quad L^q_t,\\
    \langle \Gamma(h_\lambda)\hat{\sigma}^k_\lambda ,\nabla\varphi\rangle &\to \langle \Gamma(f), \partial_k \varphi\rangle,\quad \text{in}\quad L^q_t,\\ \langle\Gamma(h_\lambda)\hat{\sigma}^k_{\lambda,i}\hat{\sigma}_{\lambda,j}^k,\partial_i\partial_j\varphi\rangle &\to \langle\Gamma(f),\Delta \varphi\rangle,\quad \text{in}\quad L^{\frac{q}{2}}_t,\\
    \langle G(h_\lambda)\Div \hat{b}_\lambda,\varphi\rangle &\to \langle G(f)\Div b,\varphi\rangle,\quad \text{in}\quad L^1_t,
  \end{aligned}
\end{equation}
as well as, $\P$ almost surely
\begin{equation}
  \begin{aligned}
    \langle G(h_\lambda)\Div \hat{\sigma}^k_\lambda ,\varphi\rangle &\to 0,\quad\text{in}\quad L^q_t\\
    \langle G(h_\lambda)(\Div \hat{\sigma}^k_\lambda)\hat{\sigma}^k_\lambda ,\nabla \varphi\rangle &\to 0,\quad\text{in}\quad L^{\frac{q}{2}}_t\\
    \langle G(h_\lambda)\partial_i\hat{\sigma}^k_{\lambda,j}\partial_j\hat{\sigma}^k_{\lambda,i},\varphi\rangle &\to 0,\quad\text{in}\quad L^{\frac{q}{2}}_t\\
    \langle H(h_\lambda)(\Div\sigma_\lambda^k)^2,\varphi\rangle &\to 0,\quad\text{in}\quad L^{\frac{q}{2}}_t.
  \end{aligned}
\end{equation}
Using the fact that $\Gamma(h_\lambda)$ is bounded by a deterministic constant, we can upgrade the $\P$ almost-sure convergence to convergence of all moments by the bounded convergence theorem.

We now have enough to pass the limit in all the terms of equation (\ref{eq:renorm-hlam}), including the stochastic integrals, proving (\ref{eq:weak-renorm-thm}) for $\Gamma(z)$ satisfying (\ref{eq:Gamma-bound}). We can then upgrade to more

\end{proof}

\section{Acknowledgments}
The author thanks Scott Smith for many fruitful discussions and for a careful reading of the manuscript.

\section{Appendix}
\begin{lem}\label{lem:diffeo-converge}
  Let $f\in L^p(\R^n)$, $p\in [1,\infty)$ and for each $\lambda>0$ let $\phi_{\lambda}:\R^n\to\R^n$ be a Lipschitz homeomorphism, satisfying
  \begin{equation}\label{eq:lipschitz-homeo-bound}
    \lim_{\lambda\to 0} \|\phi_{\lambda} - \Id\|_{W^{1,\infty}} = 0,
  \end{equation}
  then
  \begin{equation}
    f(\phi^{-1}_{\lambda}) \to f\quad \text{in}\quad L^p, \text{as }\lambda\to \infty.
  \end{equation}
\end{lem}
\begin{proof}
  It is a simple consequence of the Lipschitz convergence (\ref{eq:lipschitz-homeo-bound}) that
  \begin{equation}
    f(\phi^{-1}_\lambda) \to f,\quad\text{pointwise a.s.}
  \end{equation}
 In order to upgrade this to $L^p$ convergence, it suffices to show uniform integrability and tightness of $|f(\phi^{-1}_{\lambda}) -f|^p$. In light of the pointwise bound
  \begin{equation}
    |f(\phi^{-1}_{\lambda}) - f|^p\leqc |f(\phi^{-1}_{\lambda})|^p +|f|^p, 
  \end{equation}
  it suffices to show uniform integrability and tightness of $|f(\phi^{-1}_{\lambda})|^p$. We will also assume (taking $\lambda$ large enough) that
  \begin{equation}
    \sup_{\lambda}\|\phi_\lambda - \Id\|_{W^{1,\infty}} < \frac{1}{2}.
  \end{equation}
  This, for instance, implies the bounds
  \begin{equation}
    \left(\tfrac{1}{2}\right)^n \leq |\det\left(\nabla\phi_\lambda\right)| \leq \left(\tfrac{3}{2}\right)^n
  \end{equation}
  We will use that fact that for each $\lambda$ and measurable subset $A\subseteq \R^n$, the following inequality holds
  \begin{equation}
    \int_{A}|f(\phi^{-1}_{\lambda})|^p \leqc \int_{\phi^{-1}_\lambda(A)} |f|^p,
  \end{equation}
  where the constant $C_{p,n}$ depends on $n$ and $p$, but not $\lambda$. Note that
  \begin{equation}
    \phi^{-1}_\lambda(\{|f(\phi^{-1}_{\lambda})|^p >c\}) = \{|f|^p > c\}
  \end{equation}
  and therefore uniform integrability follows from
  \begin{equation}
    \sup_{\lambda}\int_{\{|f(\phi^{-1}_{\lambda})|^p > c\}} |f(\phi^{-1}_{\lambda})|^p \leqc \int_{\{|f|^p > c\}} |f|^p \to 0
    \end{equation}
    as $c\to \infty$. Similarly, since
  \begin{equation}
    \phi_\lambda^{-1}(\{|x|>R\}) = \{|\phi_{\lambda}(x)| > R\} \subseteq \{|x| > R - |x- \phi_\lambda(x)|\} \subseteq \{|x| > R- \tfrac{1}{2}\},
  \end{equation}
  tightness follows from
  \begin{equation}
    \sup_{\lambda} \int_{\{|x|>R\}} |f(\phi^{-1}_{\lambda})|^p \leqc \int_{\{|x| > R-\tfrac{1}{2}\}} |f|^p \to 0
  \end{equation}
  as $R\to \infty$. 
\end{proof}

\begin{lem}\label{lem:diffeo-lam-conv}
  Consider a family $\{g_\lambda\}_{\lambda >0} \subseteq L^p(\R^n)$, $p\in [1,\infty)$ satisfying
  \begin{equation}
    g_\lambda \to g \quad \text{in} \quad L^p,
  \end{equation}
  and let $\phi_\lambda:\R^n \to \R^n$ be as in Lemma \ref{lem:diffeo-converge} then
  \begin{equation}
    g_\lambda(\phi_\lambda^{-1})\to g \quad \text{in}\quad L^p. 
  \end{equation}
\end{lem}
\begin{proof}
  The proof is a simple consequence of Lemma \ref{lem:diffeo-converge} and following inequality
  \begin{equation}
      \|g_\lambda(\phi_\lambda^{-1}) - g\|_{L^p} \leq \|g_\lambda(\phi_\lambda^{-1}) - g(\phi_\lambda^{-1})\|_{L^p} + \|g(\phi_\lambda^{-1}) - g\|_{L^p}
  \end{equation}
\end{proof}

The following endpoint product-limit Lemma is more or less well-known in analysis. However, we include a proof for the reader who may not be familiar.

\begin{lem}\label{lem:prod-lim}
  Let $\{u_\lambda\}\subseteq L^\infty_{t,x}$ and $\{v_\lambda\}\subseteq L^q_t(L^p)$. Suppose that $u_{\lambda} \to u$ pointwise almost everywhere with $u \in L^\infty_{t,x}$, while $v_\lambda \to v$ in $L^q_t(L^p)$. If $\{u_{\lambda}\}$ satisfies the uniform bound
  \begin{equation}
   \sup_{\lambda} \|u_n\|_{L^\infty_{t,x}} <\infty,
 \end{equation}
 then
 \begin{equation}
   u_nv_n\to uv,\quad\text{in}\quad L^q_t(L^p).
 \end{equation}
\end{lem}
\begin{proof}
  By Egorov's theorem, for every $\ep >0$ there exists a set $A\subset [0,T]\times\R^n$, $|A|<\ep$ such that
  \begin{equation}
    u_\lambda\to u\quad\text{uniformly on }A^c.
  \end{equation}
  We may then write
  \begin{equation}
    \begin{aligned}
      u_\lambda v_\lambda - uv &= (u_\lambda - u)v + u_\lambda(v_\lambda - v)\\
      &= (u_\lambda - u)\1_{A^c}v +(u_\lambda -u)\1_{A}v + u_\lambda(v_\lambda-v).
    \end{aligned}
  \end{equation}
  Therefore
  \begin{equation}
    \begin{aligned}
      \|u_\lambda v_\lambda - uv\|_{L^q_t(L^p)} &\leq \|u_\lambda - u\|_{L^\infty(A^c)} \|v\|_{L^q_t(L^p)}\\
      &+ \left(\sup_{\lambda}\|u_\lambda\|_{L^\infty} + \|u\|_{L^\infty}\right)\|\1_{A}v\|_{L^q_t(L^p)}\\
      &+ \left(\sup_{\lambda} \|u_\lambda\|_{L^\infty}\right)\|v_\lambda - v\|_{L^q_t(L^p)}
    \end{aligned}
  \end{equation}
  Sending $\lambda \to \infty$ gives
  \begin{equation}
    \limsup_\lambda \|u_\lambda v_\lambda - uv\|_{L^q_t(L^p)} \leq \left(\sup_{\lambda}\|u_\lambda\|_{L^\infty} + \|u\|_{L^\infty}\right)\|\1_{A}v\|_{L^q_t(L^p)}.
  \end{equation}
  Sending $\ep \to 0$, and using the integrability of $v$ gives the result.
\end{proof}

\bibliographystyle{abbrv}
\bibliography{bibliography}

\end{document}